\documentclass[twoside,11pt]{article}

\usepackage{blindtext}

\usepackage[preprint]{jmlr2e}

\usepackage{hyperref}
\usepackage{amsmath}
\usepackage{amssymb}
\usepackage{mathtools}
\usepackage{amsthm}
\usepackage[capitalize,noabbrev]{cleveref}
\usepackage{algorithm}
\usepackage{algorithmic}

\usepackage{microtype}
\usepackage{graphicx}
\usepackage{booktabs}
\usepackage{url}    
           
\usepackage{amsfonts}       
\usepackage{nicefrac}            

\usepackage{xcolor}

\usepackage{lastpage}

\newcommand{\BlackBox}{\rule{1.5ex}{1.5ex}}  
\ifdefined\proof
    \renewenvironment{proof}{\par\noindent{\bf Proof\ }}{\hfill\BlackBox\\[2mm]}
\else
    \newenvironment{proof}{\par\noindent{\bf Proof\ }}{\hfill\BlackBox\\[2mm]}
\fi
\theoremstyle{plain}
 
\newtheorem{theorem}{Theorem}
\newtheorem{lemma}[theorem]{Lemma}

\newtheorem{corollary}[theorem]{Corollary}

\newtheorem{assumption}[theorem]{Assumption}

\newcommand{\e}{\varepsilon}
\newcommand{\Safe}{\mathcal S}

\newcommand{\la}{\langle}
\newcommand{\ra}{\rangle}

\newcommand{\E}{\mathbb E}
\newcommand{\R}{\mathbb R}

\newcommand{\de}{\delta}

\newcommand{\Prob}{\mathbb P}
\newcommand{\La}{\mathcal L}

\ShortHeadings{Safe Primal-Dual Optimization with a Single Smooth Constraint}{Ilnura Usmanova and Kfir Y. Levy}
\firstpageno{1}

\begin{document}

\title{Safe Primal-Dual Optimization with a Single Smooth Constraint}

\author{\name Ilnura Usmanova \email 
\\
       \addr SDSC hub at PSI,\\
       Paul Scherrer Institute,\\
       Villigen,  Switzerland 
       \AND
       \name Kfir Yehuda Levy \email 
       \\
       \addr Electrical and Computer Engineering,\\ Technion, \\
       Haifa, 
       Israel
       }

\editor{}

\maketitle

\begin{abstract}
This paper addresses the problem of safe optimization under a single smooth constraint,  a scenario that arises in diverse real-world applications such as robotics and autonomous navigation.
 The objective of safe optimization is to solve a black-box minimization problem while strictly adhering to a safety constraint throughout the learning process. Existing methods often suffer from high sample complexity due to their noise sensitivity or poor scalability with number of dimensions, limiting their applicability. We propose a novel primal-dual optimization method that, by carefully adjusting dual step-sizes and constraining primal updates, ensures the safety of both primal and dual sequences throughout the optimization. Our algorithm achieves a convergence rate that significantly surpasses current state-of-the-art techniques. Furthermore, to the best of our knowledge, it is the first primal-dual approach to guarantee safe updates. Simulations corroborate our theoretical findings, demonstrating the practical benefits of our method. We also show how the method can be extended to multiple constraints.

 
\end{abstract}


\section{Introduction}\label{section:introduction}

\paragraph{Motivation} The safe learning problem is becoming more relevant nowadays with growth of automatization, usage of reinforcement learning, and learning from interactions. In such problems, it is crucial not to violate safety constraints during the learning, even if they are not known in advance. 
As an example, imagine the task of real-time parameter tuning in manufacturing or control settings \citep{koenig2023safe, DOGRU2022107760}. The goal is to minimize production costs or maximize performance without compromising safety constraints during the real-time world interactions. 
Moreover, one might only have black-box, noisy measurements of safety and cost, and their estimated gradients, without access to their precise analytical expressions or models.  \textcolor{black}{Then, one aims to iteratively update the parameters and simultaneously take measurements, aiming to find parameters with better cost. 
It is crucial to ensure the constraints not being violated during the learning process. 
In many cases, the safety constraint is a single constraint, 
e.g., when safety must be ensured by 
limiting the probability of hitting the obstacles in robotics and reinforcement learning \citep{ray2019benchmarking}, by limiting the velocity error  in automatic controller tuning \citep{konig2023safe},  or ensuring a lower bound of the pulse energy in the automatic tuning of the Free Electron Laser \citep{kirschner2019bayesian}
.}
Such problem 
can be formulated as a \emph{safe learning} problem. Its goal is to minimise an objective $\min_x f(x)$ subject to a safety constraint $g(x)\leq 0$, given black-box information, and all queries need to be feasible during the optimization process. For such problems,  Bayesian Optimization (BO) methods provide a powerful tool for a black-box learning by construction of models based on Gaussian Processes. However, for many applications,
the parameter space might be high-dimensional, what makes BO methods  intractable due to their poor scalability with the problem dimension. In such scenarios, zeroth- and first-order methods are \textcolor{black}{preferable.} 

\paragraph{Existing work} 
There are several lines of work addressing the safe learning problem. One very powerful approach to solve global non-convex constrained optimization is Bayesian Optimization (BO). SafeOpt \citep{berkenkamp2016bayesian, sui2015safe} is the first method allowing to solve the safe learning problem. Main drawback of BO based approaches is their bad scaling with number of observations and with dimensionality. Typically, they are not applicable to dimensions higher than $d=15$ and number of samples higher than $N = 200$; since every update is extremely computationally expensive in $d$ and $N$. There are more recent approaches extending BO to higher dimensions and number of samples, such as LineBO \citep{kirschner2019adaptive, DeBlasi2020SASBOSS, Han2020HighDimensionalBO}, although they are still model based and have their limitations. In particular, LineBO is especially designed  for unconstrained problems, whereas for constrained problems it may get stuck on a suboptimal solution on the boundary. Moreover, the performance of all BO approaches depends strongly on the right choice of hyper-parameters such as a suitable kernel function, which in itself might be challenging.

Given the limitations of BO, the importance 
of first-order methods like Stochastic Gradient Descent (SGD) in machine learning becomes evident. These methods are central due to their efficiency in handling large datasets and complex models, motivating the exploration of safe first-order methods where safety constraints are paramount.
Indeed, there are several works  that explore safe learning using first-order optimization techniques. For the special case of uncertain \emph{linear constraints}, methods based on Frank-Wolfe approach were proposed in \citet{usmanova2019safe, fereydounian2020safe}, and
achieving an almost optimal rate of $\mathcal N = \tilde O(\frac{1}{\e^2})$ for convex problems,
where $\e$ is the accuracy, and $\mathcal N$ is the sample complexity. 
For \textcolor{black}{a general case} of non-linear constraints, a Log Barriers based approach was proposed in \citet{usmanova2020safe}. While being simple and general; the latter approach is unfortunately sample-inefficient 
and sensitive to noise. 
In particular, for strongly-convex, convex, and non-convex problems it reaches the rates $\mathcal N = \tilde O(\frac{1}{\e^4})$, $\mathcal N = \tilde O(\frac{1}{\e^6})$, $\mathcal N = \tilde O(\frac{1}{\e^7})$ respectively, which are 
substantially
worse than known optimal rates for non-safe problems. 

There is also a recent paper proposing the trust region technique to safe learning \citep{guo2023safe}. Such approach however has worse computational complexity, since it requires to solve quadratic approximations formulated as QCQP subproblems at every step.  Moreover, the algorithm is designed for the exact measurements, and the paper does not provide analysis for the case of stochastic measurements. 
There is also a line of work addressing \emph{online} learning  with hard constraints \citep{Yu_JMLR:v21:16-494, Guo_NEURIPS2022_ec360cb7}. However, these works guarantee in the best case $O(1)$ constraints violations, but not zero violation. More importantly, these methods require to know the analytic expression of constraint $g(x)$, 
whereas we assume that $g(x)$ is unknown and can be only accessed by a noisy oracle.


\textcolor{black}{
\paragraph{Research motivation} 
\textcolor{black}{In fully black-box optimization setting, the standard approaches of dealing with constraints, e.g., \emph{projections},} are not applicable, since the constraints are unknown in advance. 
One way to deal with this issue is replacing the original constrained problem with an unconstrained approximate, and solve it with classical unconstrained methods \textcolor{black}{like} SGD. \textcolor{black}{The current state-of-the-art approach }
 LB-SGD \citep{usmanova2020safe} 
 approximates the original problem with its log barrier surrogate. 
Its values and gradients grow to infinity close to the boundary, \textcolor{black}{automatically pushing the iterates away from the boundary, thus, ensuring safety. } 
However, in the presence of noise, the noise in the log barrier gradient estimators also amplifies close to the boundary. 
In order to guarantee convergence and safety of the updates near the boundary, the log barrier approach requires $O(\frac{1}{\e^4})$ measurements per iteration, which leads to sample-inefficiency}.
\textcolor{black}{
Then, the question raises: \emph{Can we use an alternative approach, 
e.g., utilize the Lagrangian function of the original problem as an unconstrained surrogate, while still guaranteeing safety?
} 
Appealingly, 
the Lagrangian gradients
are much more robust to noise compared to log barrier gradients. }
\textcolor{black}{\paragraph{Challenge in safety of dual approaches.}  
Lagrangian duality allows to reformulate the original problem 
$\min f(x)$ s.t. $g(x) \leq 0$ 
as a min max problem: $\min_x\max_{\lambda \geq 0} \mathcal L(x,\lambda) $ of the  Lagrangian $\mathcal L(x,\lambda) := f(x) + \lambda^T g(x)$, with dual vector $\lambda \in \R^m$  corresponding to $m$ inequality constraints. 
Primal-dual approaches replace the primal problem  $\min_x \max_{\lambda\geq 0} \mathcal L(x,\lambda) $ by its dual problem over $\lambda$s : $\max_{\lambda \geq 0} \min_x \mathcal L(x,\lambda) = \max_{\lambda \geq 0}\mathcal L(x_{\lambda},\lambda),$ where the corresponding to $\lambda$ primal variable $x_{\lambda}$ is a solution of an unconstrained problem $\arg\min_{x}\mathcal L(x,\lambda)$. 
 The updates are done in both primal and dual spaces $(x, \lambda)$. 
 }

\textcolor{black}{
Unfortunately, general primal-dual approaches do not guarantee feasible primal updates, and therefore fail to ensure safety. 
The dual feasibility set  $\lambda \geq 0$ does not correspond to the primal feasibility set. 
When updating the dual variable $\lambda$, the corresponding primal variable $x_{\lambda}$ can be \textcolor{black}{either} feasible or not, depending on the value of $\lambda$. 
E.g., for $\lambda = 0$, which is still a feasible dual variable, the corresponding $x_{\lambda}$ is a minimizer of an unconstrained problem $x_{\lambda} = \arg\min_x f(x)$, which is most likely infeasible in the primal space. }
\textcolor{black}{
The question is: \emph{Can we enforce the dual steps to stay within a primal feasibility region?} In this paper, we show that the answer is \emph{Yes}: 
in the case of a single smooth constraint, the primal feasibility set in the dual space takes a simple shape:  $\{\lambda\in \R_+:\lambda \geq \lambda^*\}$. 
\textcolor{black}{Then, we introduce two additional mechanisms: 
\textbf{(i)}
to ensure the dual updates $\{\lambda_t\}_t$ always reside in this set,  and 
\textbf{(ii)}
to ensure \emph{safe transit} for consecutive corresponding primal variables, i.e.,
from $x_{\lambda_t}$ to $x_{\lambda_{t+1}}$. }
}

\paragraph{Our contribution}
In this work, for the special case of a single smooth safety constraint, 
\textcolor{black}{we propose a new approach for safe learning using duality. By iteratively updating dual variables and solving primal subproblems, our method ensures safety for all primal and dual iterates. 
In particular, by starting from a large enough dual variable $\lambda_0$, and restricting the dual step-size, it ensures the primal safety for every dual update $\lambda_{t+1}$, i.e., that the corresponding primal Lagrangian minimizer $x_{\lambda_{t+1}}$ is feasible. \textcolor{black}{We can find such a safe dual step-size when the dual function is smooth, that holds for a strongly-convex primal problem. } 
Moreover, we ensure that $x_{\lambda_{t+1}}$ lies in a safety set $\mathcal S(x_t)$ of a previous primal variable  $x_{\lambda_{t}}$, forming a chain of safe dual updates. See \Cref{fig:1} for the illustration. 
Therefore, by restricting the primal steps to stay within this safety region, it guarantees safety for all the primal updates, converging from  $x_{\lambda_{t}}$ to $x_{\lambda_{t+1}}$.  \textcolor{black}{We also extend this approach to non-convex case, by solving a sequence of regularized strongly-convex problems.  }}

\begin{figure}
    \centering    \includegraphics[width=0.6\linewidth]{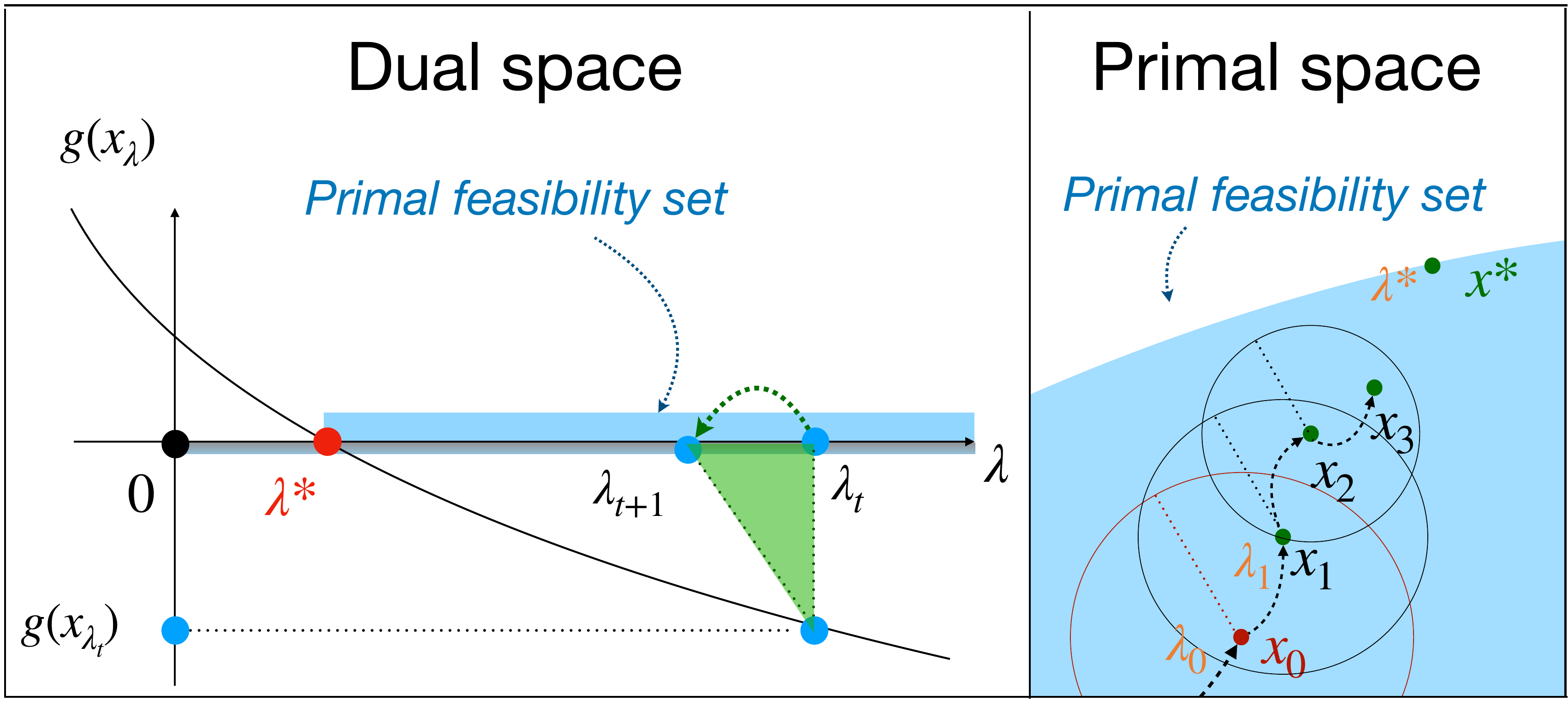}
    \caption{Primal feasibility set in primal and dual spaces,
    for  problem $\min_x f(x), \text{ s.t. }g(x)\leq 0$. 
    For primal feasibility of dual updates, we restrict the dual step sizes. For feasibility of the primal updates, we ensure the next dual update $x_{t+1}$ lies within a safety set of $x_t$ (circle).
    Here, we denote $x_{\lambda_t}$ as $x_t$ for simplicity.}
    \label{fig:1}
\end{figure}
As shown in Table~\ref{tab:comparison}, our method enjoys a \emph{substantially better sample complexity} compared to current state-of-the-art. Concretely, in the case of a single smooth safety constraint, we achieve a sample complexity as $\mathcal N = \tilde O(\frac{1}{\e^2})$ for strongly-convex problem with stochastic gradient feedback, $\mathcal N = \tilde O(\frac{1}{\e^4})$ for convex problems, and $\mathcal N = \tilde O(\frac{1}{\e^6})$ for non-convex problems, where under $\tilde O(\cdot)$, we hide a multiplicative logarithmic factor. 
Despite the main setting of a single constraint being limiting, we can extend our approach to multiple constraints by taking a maximum over constraints and applying any smoothing procedure to the resulting constraint. Such an approach leads to a worse sample complexity by a factor of $\frac{1}{\e}$; however, in convex and strongly convex cases, it is still strictly better than the baseline. We provide a more detailed discussion in Section \ref{section:extension}.  The corresponding rates can be found in Table~\ref{tab:comparison}. 

\emph{An advantage} of our approach is that it is \emph{generic}: it allows for use of \emph{any constrained optimization solver}  $\mathcal A$ on user's choice for given type of feedback, for the internal primal steps to find $x_{t+1}$ given the fixed $\lambda_{t+1}$. 
\emph{Another limitation} is that for convex case it is a two-cycle method (three-cycle for non-convex problems), which makes it harder to extend for online problems. Despite that, 
we believe our method has its own theoretical value, as it  
is conceptually new;  feasibility during the learning was never guaranteed before for primal-dual methods.


\begin{table}
\caption{ Sample complexity of first-order safe smooth optimization algorithms:  SafeOpt \citep{berkenkamp2016bayesian}, LB-SGD \citep{usmanova2023log}, and SafePD (this work).     
    Here $\e$ is the target accuracy, 
    and $d$ is the dimensionality. NLP stands for Non-Linear Programming. 
    In  SafeOpt, $\gamma(d)$ depends on the kernel, and might be exponential in $d$. 
    The last column shows the complexity of a single iterate, for SafeOpt it requires solving a non-linear optimization sub-problem. SafePD$^*$ is the extension to multiple constraints.}
    \label{tab:comparison}
\begin{center}
\begin{tabular}
{llllll}
\textbf{Method} 
&  \textbf{Str-cvx}  
& \textbf{Cvx} 
& \textbf{Non-cvx} 
&\textbf{Step complexity} 
\\
\hline \\
SafeOpt 
         & -
         & -
         & $O\left(\frac{\gamma(d)}{\e^2}\right)$ &  NLP subproblem \\ 
 \hline \\        
LB-SGD
         & $\tilde O\left(\frac{1}{\e^5}\right)$ &
         $\tilde O\left(\frac{1}{\e^6}\right)$
         & $ \tilde O\left(\frac{1}{\e^7}\right)$ & gradient step\\
\hline \\
 \bf SafePD (1 constraint) 
         & $\tilde O\left(\frac{1 
         }{\e^2}\right)$ &  $\tilde O\left(\frac{1
         }{\e^4}\right)$ & $ \tilde O\left(\frac{1
         }{\e^6}\right)$ & gradient step\\
         \hline
         \\
\bf SafePD$^*$
         & $\tilde O\left(\frac{1
         }{\e^3}\right)$ &  $\tilde O\left(\frac{1
         }{\e^5}\right)$ & $ \tilde O\left(\frac{1
         }{\e^7}\right)$ & gradient step\\
         \hline
\end{tabular}
\end{center}
\end{table}

\section{Problem Formulation
}\label{section:problem}

In this paper, we consider a \textit{safe learning} problem:
    \begin{align}\label{problem}
        \min_{x\in \R^d}~ & f(x) \tag{$\mathcal P$}\\
        \text{s.t. } & g(x) \leq 0, \nonumber
    \end{align}
\looseness -1 where the objective $f: \R^d \rightarrow \R$ and the constraint $g: \R^d \rightarrow \R$ are unknown smooth possibly non-convex functions, \textcolor{black}{and only can be accessed via a black-box oracle.}  We denote by $\mathcal X$ the feasible set $\mathcal X:= \{x\in \R^d: g(x)\leq  0\}.$  Crucially, we require that all iterates and query points are  feasible, i.e., $x_t \in \mathcal X$. 
Importantly, since $g(\cdot)$ and, therefore, the set $\mathcal X$ are unknown, we need to  ensure \emph{safety while exploring and learning the constrained optimizer.} \textcolor{black}{We define the corresponding Lagrangian function by $\mathcal L(x,\lambda) := f(x) + \lambda g(x)$.
   }
\paragraph{Feedback}
We assume $f(\cdot)$ and $g(\cdot)$ can be accessed via a noisy first-order oracle $\mathcal O(\cdot)$ which for any query $x\in\mathcal X$ returns a pair of stochastic measurements of $f$ and $g$ and their gradients respectively
\textcolor{black}{$(F(x), \nabla F(x))
$ and $(G(x), \nabla G(x))
$. We assume that these measurements are unbiased and  $\sigma$- and $\hat\sigma$ sub-Gaussian respectively, that is, $\forall x\in \mathcal X$ \begin{align*}
& \E\|F(x) - f(x)\|^2\leq \sigma^2, ~~~~~~~ \E\|G(x) - g(x)\|^2\leq \sigma^2,\\
& \E\|\nabla F(x) - \nabla f(x)\|^2\leq \hat \sigma^2,~  \E\|\nabla G(x) - \nabla g(x)\|^2\leq \hat \sigma^2,
\end{align*} where \textcolor{black}{by $\|\cdot\|$ we denote $\ell_2$-norm on $\R^d$.}}  We call a random variable $\xi$ zero-mean $\sigma$-sub-Gaussian if  
    $   
        \E\left[ e^{\omega \xi}\right] \leq \text{exp}\left(\frac{\omega^2\sigma^2}{2}\right) ~ \forall \omega \in \R. 
    $ 
It can also be shown using Taylor expansion that
$\mathbb E\|\xi\|^2 \leq \sigma^2$. We also assume that when we query the oracle  multiple times, even at the same $x\in \mathcal X$, the resulting randomizations are independent of each other. 





\paragraph{Assumptions} 
A function $f: \R^d\rightarrow \R$ is called \textit{$L$-Lipschitz continuous} if
    $|f(x) - f(y)|\leq L\|x - y\|.$
It is called \textit{$M$-smooth} if 
    $f(x) \leq f(y) + \la \nabla f(y), x-y\ra + \frac{M}{2}\|x-y\|^2.$ It is called \textit{$\mu$-strongly-convex} if 
    $f(x) \geq f(y) + \la \nabla f(y), x-y\ra + \frac{\mu}{2}\|x-y\|^2.$ 
Throughout the paper we assume:
\begin{assumption}\label{assumption:smoothness} $ f(x), g(x)$ are both $M_f$- and $M_g$-smooth respectively, $g(x)$ is also $L_g$-Lipschitz continuous on $\mathcal X.$ 
\end{assumption}
\begin{assumption}\label{assumption:safe_init}
    There exists a known feasible starting point $x_0\in \mathcal X$ such that $-g(x_0)\geq \alpha >0.$ 
\end{assumption}
Without a safe point $x_0$, even the initial measurements can be unsafe. 
The above directly implies the Slater's condition in the convex case. By $\alpha>0$ we denote the lower bound on the absolute value of the constraint at $x_0$: $\alpha \leq -g(x_0).$ 
\begin{assumption}\label{def:beta} 
There exists feasible $\tilde x\in\mathcal X$ such that $-g(\tilde x)\geq \beta$ for some  known $\beta>0$. If maximum value of $-g(x)$  exists, set $\beta := \max_{x\in\R^d} \{-g(x)\}.$ Note that by definition, $\beta \geq \alpha.$ 
\end{assumption}

\paragraph{Optimality criteria and sample complexity}
    For problem (\ref{problem}), we call $ x$ an $\e$-approximate feasible solution, if $f( x) - f(x^*) \leq \e$ and $g( x) \leq 0$. We call pair $( x,  \lambda)$ an $(\e_{p},\e_c)$-approximate KKT point, if
    the following holds: 
\begin{align}
     \| \nabla_x \mathcal L(x,\lambda)\| & \leq \e_{p}, \tag{$(\e_p,\e_c)$-KKT.1}\\
    - g(x)\lambda & \leq \e_c,\tag{$(\e_p,\e_c)$-KKT.2}\\
      \lambda \geq 0, -g(x) & \geq 0,\tag{$(\e_p,\e_c)$-KKT.3}
\end{align}
where $ \e_{p}$ is the accuracy in the Lagrangian gradient norm, and $ \e_c$ is the accuracy of the approximate complementarity slackness. 
 For any algorithm $\mathcal A'$, by $\mathcal N_{\mathcal A'}
    (\e)$ we define a total sample complexity of $\mathcal A'$ to solve a given optimization problem up to accuracy $\e$, \textcolor{black}{which is a total number of oracle queries including complexity of internal algorithms. }
    

\section{Preliminaries}
\textcolor{black}{In this section, we introduce particular properties of dual function and safety sets,  crucial for our analysis. }
\paragraph{Duality}
In the convex case, the original constrained problem is equivalent to the dual problem 
    $\max_{\lambda\geq 0} \min_{x\in \R^d}\La(x,\lambda) = \max_{\lambda\geq 0} d(\lambda).$
In the above we denote by $d(\cdot)$ the dual function $d(\lambda):=\min_{x\in  \R^d}\La(x,\lambda).$ The corresponding to $\lambda$ primal variable  is denoted by $x_{\lambda}:=\arg\min_{x\in \R^d}\La(x,\lambda).$ Let $\lambda^*:=\arg\max_{\lambda\geq 0} d(\lambda)$ be the optimal dual variable. 
 Then, we can upper bound the norm of the optimal $\lambda^*$:
\begin{lemma}\label{lemma:upper_lambda}
    Let Assumption \ref{def:beta} hold for (\ref{problem}), 
    \textcolor{black}{ then $\lambda^* \leq \Lambda := \frac{\Delta_f}{\beta}$ , where $f(x)- f(x^*) \leq \Delta_f$ for all $x \in \mathcal X$. Additionally, $\lambda^* \leq \frac{f(x_0) - f(x^*) }{-g(x_0)}$ holds.}
    %
\end{lemma} For the proof see Appendix \ref{proof:lemma:upper_lambda}. \textcolor{black}{In the above, $\Delta_f$ is a general upper bound on $f(x) - f(x^*)\forall x \in \mathcal X.$ Note that if $\Delta_f$ is unknown, it can still be upper bounded by $\Delta_f := \max_{x\in\mathcal X} f(x) - f(x^*) := f(\hat x) - f(x^*) \leq \|\nabla f(\hat x)\| R + \frac{M_f}{2}R^2,\text{with } \hat x \in \arg\max_{x\in \mathcal X} f(x),$ where  $\|\nabla f(\hat x)
\| \leq \|\nabla f(x_0)\| + M_f R $ due to the smoothness of $f$. Here, by $R$ we denote an upper bound on the initial distance to the set of optimal solutions:  $\|x_0 - x^*\|\leq R~, ~ \forall x^* \in \arg\min_{x\in \mathcal X}f(x).$ }

Also, the dual problem 
of \ref{problem} 
with a strongly-convex objective 
$f$ is smooth:
\begin{lemma}\label{lemma:dual-smoothness}
    If $f(x)$ is $\mu_f$-strongly-convex, and $g(x)$ is convex and $L_g$-Lipschitz-continuous, then the dual function $d(\lambda)$ is $\frac{2 L_g^2}{\mu_f}$-smooth. 
\end{lemma}
For the proof see Appendix \ref{proof:lemma:dual-smoothness}. 
\citet{Yu2015OnTC} proved that under some rank conditions on the constraints Jacobian, the dual function is strongly-concave with some $\mu_d$. (See Theorem 10, \citet{Yu2015OnTC}).
For a single constraint the rank condition holds automatically. Moreover, we can prove the local strong-concavity and provide a bound on $\mu_d$ below.
\begin{lemma}
\label{lemma:local_dual_strong_concavity}
Problem (\ref{problem}) has a    \emph{$\mu_d$-locally strongly-concave} dual function with $    \mu_d = \frac{l^2}{M_f+\lambda M_g}$
for all $\lambda$ such that  $\|\nabla g(x_{\lambda})\|\geq l>0$. For convex $g$, it implies
    $\mu_d \geq  \frac{\beta^2}{4R^2(M_f+\lambda M_g)}$
for all $\lambda$ such that  $g(x_{\lambda})\geq -\beta/2$, where $\beta$ is defined in Definition \ref{def:beta}.
\end{lemma}
The proof is based on the definition of the Hessian of the dual problem, and can be found in Appendix \ref{proof:lemma:local_dual_strong_concavity}. The local strong concavity and smoothness of the dual problem allows to bound the convergence rate of our approach, where outer iterations are basically dual gradient ascent steps.
\paragraph{Safety set} 
Here, we introduce an notion of a \textit{safety region} at the current point $x_t$, \textcolor{black}{using Lipschitz continuity of the constraint.} 
Let $\mathbb B^d(c,r)
$ denote the Euclidean ball in $\R^d$ centered at $c$ with radius $r>0$. 
For point $x_{t}\in \mathcal X$ the \emph{safety region} $\Safe(x_t) := \mathbb B^d(x_t, \frac{-g(x_t)}{\textcolor{black}{2}L_g})$ is the ball with center at $x_t$ and radius $\frac{-g(x_t)}{\textcolor{black}{2}L_g}$, where $L_g$ is the Lipschitz continuity constant of $g(x)$ on $\mathcal X$. 
 \begin{lemma}\label{lemma:safety_set}
   Given that $x_t \in \mathcal X$ is strictly feasible, all points $x$ within the safety set $\mathcal S(x_t)$ satisfy $g(x) \leq \frac{g(x_t)}{2} <0 $, implying their strict 
   feasibility.
\end{lemma}
\begin{proof}
    Indeed, from Lipschitz continuity we have for all $x\in\mathcal S(x_t)$ : $g(x) - g(x_t) \leq L_g\|x - x_t\|$, what implies 
        $g(x) \leq   g(x_t) + L_g\|x - x_t\| \leq g(x_t) + L_g\frac{-g(x_t)}{\textcolor{black}{2}L_g}\leq \frac{g(x_t)}{2}.$
\end{proof}
\textcolor{black}{Often, exact measurements of $g(x_t)$ are unknown, in this case we can use the upper bound $\hat g(x_t)$ with confidence $1-\delta$ and define the corresponding \emph{approximate safety set} by $\hat\Safe(x_t) := \mathbb B^d(x_t, \frac{-\hat g(x_t)}{L_g})$, that guarantees feasibility of $x\in  \hat\Safe(x_t)$ with probability $1-\de.$
}

\section{Strongly-Convex Problem}
As the main building block of our approach, we propose a primal-dual safe method addressing strongly-convex and smooth optimization problem. Strongly-convex problem in our paper satisfies the assumption below: 
\begin{assumption}\label{assumption:strong-conv}
    The objective $f(x)$ is $\mu_f$-strongly-convex and $M_f$-smooth for some $M_f\geq\mu_f>0$. 
    Also, the constraint function $g(x)$ is convex and $M_g$-smooth. 
\end{assumption}
    
 Then, the main idea of our approach can be described as follows. 
\emph{First,} we find an initial primal-dual pair $(\check \lambda, \check x)$, such that $\check x\approx \min_{x\in\R^d}\mathcal L(x,\check \lambda)$ is strictly feasible. 
By taking $\check \lambda$ big enough, we guarantee that any descent method \textcolor{black}{$\check{\mathcal A}$} applied to $\arg\min_{x\in\R^d}\mathcal L( x,\check \lambda)$ starting from $x_0\in Int(\mathcal X)$ 
does also imply feasibility of its  iterates, including feasibility of $\check x.$ \textcolor{black}{We set $(x_1,\lambda_1) = (\check x,\check \lambda)$.} 
\emph{Next,} we iteratively update pair $(\lambda_t, x_t)$ as follows: 

1.  We iteratively decrease the dual variable \textit{(dual iteration)} $\lambda_{t+1} \leftarrow \max\{\lambda_{t} - \frac{-\mu g(x_t)}{\textcolor{black}{8}L_g^2},0\}$. This update corresponds to the dual gradient ascent with the step-size $\gamma = \frac{\mu }{\textcolor{black}{8}L_g^2},$ which guarantees that the corresponding primal solution $ x_{\lambda_{t+1}}:=\arg\min_{x\in\R^d} \La(x,\lambda_{t+1})$ 
    lies in the \emph{approximate safety region} $\hat \Safe(x_t)$ around the last update $x_t$. 
    
2. We solve the primal problem \emph{(primal iteration)} of the corresponding dual variable $x_{t+1} \leftarrow \arg\min_{x\in\hat{\mathcal S}(x_{t})} \mathcal L(\cdot, \lambda_{t+1})$ constrained to the safety region, up to accuracy \textcolor{black}{$\eta_t>0$}, i.e., $\mathcal L(x_{t+1},\lambda_{t+1}) - \mathcal L(x_{\lambda_{t+1}},\lambda_{t+1})\leq \eta_t.$  This can be done safely using any projection descent method, or any approach that preserves feasibility subject to a ball constraint $\hat{\mathcal S}(x_{t})$. 

This way, 
we guarantee a continuous sequence of primal subproblems such that the path between the new $x_{t+1}$ and the past $x_t$ lies fully in the safety set of $x_t$. 
Our approach for the Strongly-Convex case is depicted in \Cref{alg:base}. \textcolor{black}{We set the required primal accuracy to be 
    $\eta_{t} := \begin{cases}\frac{\mu_f \hat g(x_{t})^2}{128 L_g^2},
    &\text{ if } -\hat g(x_{t})\lambda_{t+1} > \e_c\\ 
    \min\{\frac{\mu_f}{M_L}\e_p^2, \frac{\e}{2}\}, &\text{ if } -\hat g(x_{t})\lambda_{t+1}\leq \e_c\end{cases}$ and prove it is enough in the following sections. The second case corresponds to the accuracy of the last primal step $\eta_T.$}
\textcolor{black}{
    In the above, by $\eta_t$ we denote accuracy of solving the primal subproblem at step $t$.
    By $\check \eta$ we denote accuracy of solving the initial primal subproblem.
    }
\begin{algorithm}[h]
\caption{Strongly-Convex Safe Algorithm (SCSA)}\label{alg:base}
\begin{algorithmic}
\item[\emph{Initialization:}]  $x_0: -g( x_0) \geq \alpha > 0,$ $\beta>0$,  $L_g,$ $\mu_f$, 
 $M_f, M_g >0,$ $\Delta_f>0,$ $\{\epsilon_t\}_{t=0}^T$, $\{\eta_t\}_{t=0}^T$. Set $\check \lambda \leftarrow \frac{\Delta_f}{\alpha}.$
    \item[0.]\emph{ Preliminary step:}
    $\check x \leftarrow \check{\mathcal A}
    (\mathcal L(\cdot, \check \lambda),\R^d,  x_{0})$ (up to acc. $\check \eta \leq \frac{\mu_f \alpha^2}{8L_g^2}$), and set $(x_1, \lambda_1) \leftarrow (\check x, \check \lambda)$
    \FOR{$t\in\{1,\ldots,T\}$ }
    \item[1.] \emph{ Estimate $\epsilon_t$-approximate upper bound on $g(x_{t})$  using a minibatch with $n_t := \frac{4\sigma^2}{\epsilon_t^2}\ln\frac{T}{\de}$ samples:}\\
    \textcolor{black}{$$\hat g(x_{t}):=\frac{\sum_{i=1}^{n_t} G_i(x_t)}{n_t} + \frac{\sigma}{\sqrt{n_t}}\sqrt{\ln\frac{T}{\de}}$$
   } 
    \item[2.] \emph{ Safety set:} $\hat \Safe(x_{t}) \leftarrow \mathbb B^d(x_{t}, \frac{-\hat g(x_{t})}{L_g})$
    \item[3.] \emph{ Dual update:} $\lambda_{t+1} \leftarrow \max\{\lambda_{t} + \gamma\hat \nabla d(\lambda_{t}),0\}$, 
    where $\gamma = \frac{\mu_f}{\textcolor{black}{8}L_g^2}$ and $\hat \nabla d(\lambda_t) = \hat g(x_{t})$ 
    \item[4.]
    \emph{ Primal update:} $x_{t+1} \leftarrow \mathcal A
    (\mathcal L(\cdot, \lambda_{t+1}),\hat{\mathcal S}(x_{t}), x_{t})$, up to accuracy $\eta_{t}$
    \item[5.] \textbf{If} $-\hat g(x_{t})\lambda_{t+1} \leq \e_c$  \textbf{break}
    \ENDFOR
    \item[\emph{Output:}]  $(x_T,\lambda_T)$
\end{algorithmic}
\end{algorithm}
\vspace{-2mm}
\paragraph{Internal algorithms} 
     By algorithm $\mathcal A(f,\mathcal C, x_0)$ we denote an optimization procedure with known simple constraint set $\mathcal C$, for minimizing $\min f(x)$ s.t. $x\in \mathcal C$, where $f$ is $M$-smooth and $\mu$-strongly-convex, the procedure is starting from $x_0\in \mathcal C$ and all the updates are feasible $x_t\in \mathcal C$. For internal  algorithm $\mathcal A$ we can use any algorithm \emph{ensuring feasible iterations} with known feasibility set $\hat {\mathcal S}(x_{t})$,  e.g., the projected gradient descent: $
\mathcal A_{ PGD}(\mathcal L(\cdot,  \lambda_{t+1}),\hat{\mathcal S}(x_{t}), x_{t})$ or  projected stochastic gradient descent $\mathcal A_{SGD} (\mathcal L(\cdot,  \lambda_{t+1}),\hat{\mathcal S}(x_{t}), x_{t}).$ Adam or other momentum-based algorithms can be used too.  By $\check{\mathcal A}$ we denote an optimization method 
ensuring \emph{descent} at every step.

\paragraph{Estimating a constraint value.} \textcolor{black}{Note that in order to estimate the safety set at point $x$, we need to estimate the constraint value $g(x)$. The accuracy of this estimation is limited by the noise level of measurements. We can decrease this error by taking a minibatch of several noisy samples of $g(x)$ at the required point $x$. Let us denote by $ \epsilon_t$ the required approximation accuracy of the constraint value estimation at iteration $t$  with high probability $\Prob\{\forall t \leq T: g(x_t) - \hat g(x_t) \leq \epsilon_t\}\geq 1-\de.$ As we later show in Theorem \ref{theorem:str_cvx_conv}, estimating constraint values $g(x_t)$ up to accuracy $\epsilon_t = O(\e)$ is enough for safe convergence to $\e$-approximate solution.  } 
In \emph{Step 1.}, we approximate $g(x_t)$ by using a minibatch of $n_t$
samples.
Note that from standard concentration inequalities for sub-Gaussian random variables, we have  
$\Prob\left\{\left|\frac{ \sum_{i=1}^{n_t} G_i(x_t)}{n_t}-g(x_t)\right|\leq \frac{\sigma}{\sqrt{n_t}}\sqrt{\ln\frac{1}{\de}}\right\}\geq 1-\de$. 
Therefore, by taking $n_t \geq \frac{\sigma^2}{\epsilon_t^2}\ln\frac{T}{\de}$ ensures that  $\Prob\{\forall t \leq T: g(x_t) - \hat g(x_t) \leq \epsilon_t\}\geq 1-\de.$  
 
    In the next three subsections, we provide the theoretical justification of our procedure. In particular, we provide the safety and convergence guarantees, and bound the sample complexity.


\subsection{Safety}
We split this section into two parts, describing separately safety of finding the initial pair $(\check x, \check \lambda)$ and safety of the transition $(x_t, \lambda_t) \rightarrow (x_{t+1}, \lambda_{t+1}).$

\paragraph{Safe dual initialization} First, we prove the safety of obtaining $(\check x, \check\lambda)$ in a safe way starting from $x_0$. For that, one can set  $\check\lambda \geq  \frac{\Delta_f}{\beta}$, and use any \emph{descent} method to minimize $\min_{x\in\R^d}\mathcal L(x,\check\lambda).$ 

\begin{lemma}\label{lemma:safe_init}
Let $\check \lambda = \frac{\Delta_f}{\alpha} \geq \frac{f(x_0) - \min_x f(x)}{-g(x_0)}.$
Then, all iterates $x_t$ of any descent method for $\mathcal L(x,\lambda)$ starting at $ x_0$ would guarantee feasibility of the iterates $g(x_t) \leq 0.$
\end{lemma} For the proof see Appendix \ref{proof:lemma:safe_init}. 
Specifically, we can use SGD with an appropriate minibatch size to estimate the gradients well enough, to guarantee descent at every step with high probability in the following \Cref{lemma:init_sgd}. For its proof see Appendix \ref{proof:lemma:init_sgd}.

  \begin{lemma}\label{lemma:init_sgd}
      Let us use mini-batch SGD $\check{\mathcal A}_{mSGD}$ during the initial stage of SCSA (Alg. \ref{alg:base}) with properly chosen mini-batch size $n_0(\tau) = \frac{4\sigma^2}{\|\nabla \tilde{\mathcal  L}(x_{\tau-1}, \check\lambda)\|^2}$ 
      and step-size $\gamma_{\tau}\leq \frac{1}{M_f + \check\lambda M_g}.$ 
      Then, it ensures descent at every step, 
      and converges to $\check\eta$-optimal point after $\mathcal T = O(\log \frac{1}{\check\eta})$ iterates, with total sample complexity $\mathcal N_0 = \tilde O(\frac{1}{\check\eta^2})$.
  \end{lemma}

        \paragraph{Safe transition} \textcolor{black}{Next, we discuss a safe transition from $(x_t,\lambda_t)$ to $(x_{t+1},\lambda_{t+1})$, given that $(x_t,\lambda_t)$ are feasible. 
        Consider the dual problem $d(\lambda)$, where $\lambda$ is scalar in the case of a single constraint.  
        As shown at \Cref{fig:2}, the optimal $\lambda^*$ corresponds to the maximum of $d(\lambda)$, i.e., $ \|\nabla d(\lambda^*)\| = 0$ if the solution is on the boundary, otherwise $\lambda^* = 0$. 
        Recall that $\|\nabla d(\lambda)\| = g(x_{\lambda})$, which should be negative for $ x_{\lambda} \in \mathcal X.$  Concavity of $d(\lambda)$ implies monotonously decreasing $\nabla d(\lambda).$ 
        Therefore, indeed, $\lambda \geq \lambda^*$ corresponds to the primal feasibility region in the dual space. Secondly, note that $\nabla d(\lambda)$ growth is limited by the smoothness bound of $d(\lambda). $ Hence, we can control that $g(x_{\lambda_t})$ stays in the negative region by controlling the dual step size. 
        But, just guaranteeing the next $x_{\lambda_t}$ is safe is not enough, since we should also find a way to safely transit from $x_t$ to $x_{t+1}$. In order to do that, we restrict the step size even more, to guarantee that $x_{\lambda_{t+1}} \in \hat {\mathcal S}(x_t).$  }
        \begin{figure}
            \centering
\includegraphics[width=0.6\linewidth]{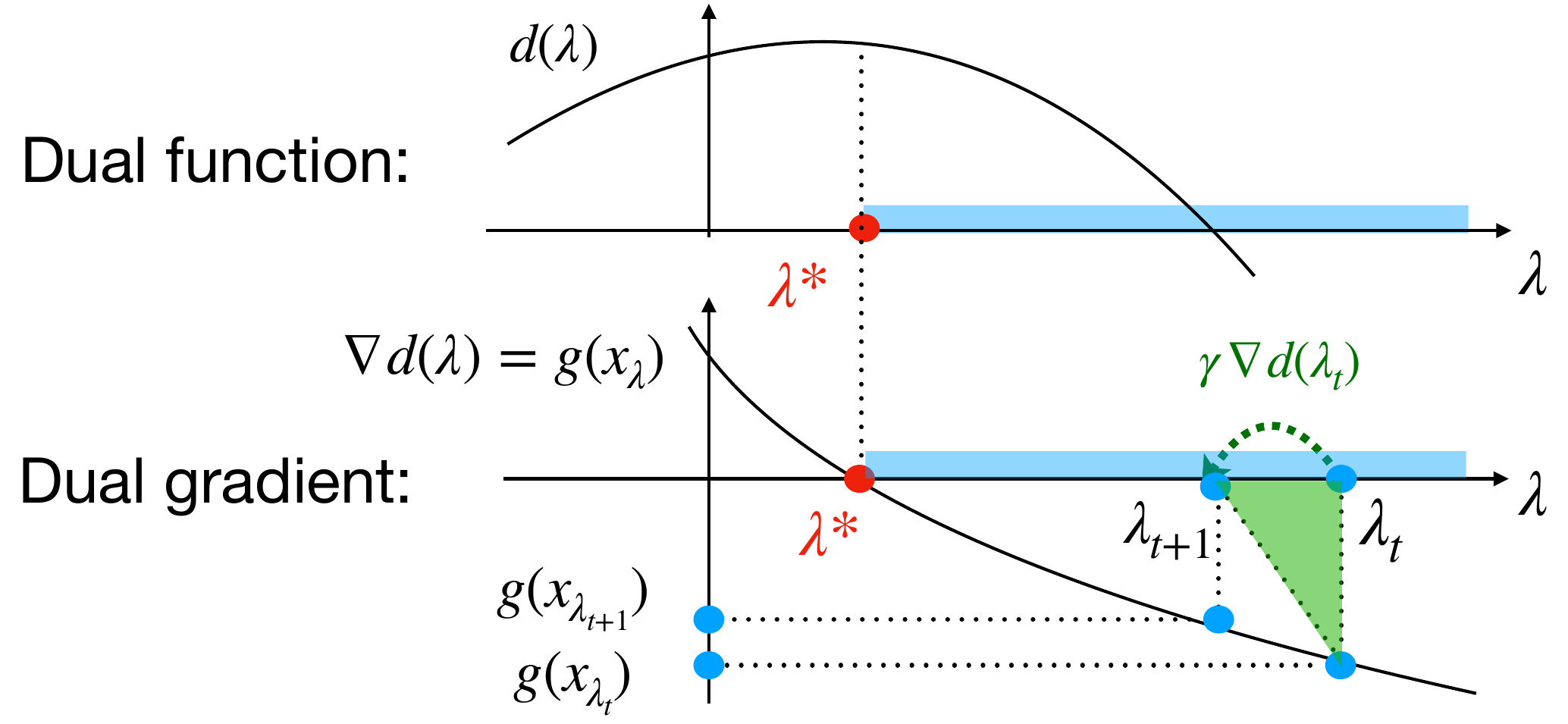}
\caption{Primal feasibility of the dual iterates}
            \label{fig:2}
        \end{figure}

\begin{lemma}\label{lemma:dual_upd}
    \textcolor{black}{Let $\lambda_t$ correspond to a feasible point $x_{\lambda_t}.$ Let the dual step } be bounded by  $\lambda_t - \lambda_{t+1} \leq \frac{-\mu_f \hat g(x_t)}{\textcolor{black}{8}L_g^2}$, as defined in Step 3 of Alg. \ref{alg:base}.
    Also, let $x_{t}$ be an $\eta_{t}$-accurate solution of 
    $\min_{x\in\R^d} \mathcal L(x,\lambda_t)$ with $ \eta_{t} \leq  \frac{\mu_f (-\hat g(x_{t-1}))^2}{\textcolor{black}{32} L_g^2}.$
    Then, $\| x_{\lambda_{t+1}} - x_t\| \leq \frac{-\hat g(x_t)}{\textcolor{black}{2}L_g}$,
    that is, $x_{\lambda_{t+1}} \in \hat{\mathcal S}(x_t)$ with probability $1-\de$.   
\end{lemma}
The full proof can be found in Appendix \ref{proof:lemma:dual_upd}. This lemma combined with the way the primal updates are done 
$x_{t+1}\leftarrow \arg\min_{x\in\hat S(x_t)}\mathcal L(x,\lambda_{t+1})$ directly implies safety of SCSA (Alg. \ref{alg:base}). 
  Indeed, by using a feasible algorithm $\mathcal A(\mathcal L(\cdot, \lambda_{t+1}),\hat{ \mathcal S}(x_t), x_t)$ with projections onto the safety set $ \hat{\mathcal S}(x_t)$, we ensure that all the inner iterations are safe, including the last one $x_{t+1}$. And for the very first initialisation run of algorithm $\check{\mathcal A}$, the updates are feasible as long as $\check{\mathcal A}$ is a descent algorithm. This implies 
  our safety result:
\begin{theorem}
    Under Assumptions \ref{assumption:safe_init}, \ref{assumption:smoothness}, \ref{def:beta}, and \ref{assumption:strong-conv}, SCSA (Alg. \ref{alg:base})  guarantees that with probability $\geq 1-\delta$, 
    all oracle queries for iterates 
    $t = 1,\ldots, T$ are safe, i.e, 
    $g(x_t)\leq 0$ and $g(x_{t,\tau})\leq 0$ for all steps $\{x_{t,\tau}\}_{\tau= 1,\ldots,\mathcal{T}}$ of internal algorithms $\mathcal A$ and $\check{\mathcal A}$.
\end{theorem}
Moreover, note that $x_{t+1} \approx  \arg\min_{x\in \hat{\mathcal S}(x_{t})} \mathcal L(x,\lambda_{t+1}) = \arg\min_{x\in\R^d} \mathcal L(x,\lambda_{t+1}) $ since $x_{\lambda_{t+1}}\in \hat{\mathcal S}(x_{t})$ by Lemma \ref{lemma:dual_upd}. 
 The above fact allows us to guarantee that our safe primal updates restricted to a safe region still form a sequence of primal-dual pairs $(x_t, 
\lambda_t)$. 

\subsection{Sample complexity}
In this section, we derive the sample complexity of the SCSA (Alg. \ref{alg:base}). \textcolor{black}{For that, we analyze the number of outer iterates, and sample complexity of the inner algorithms. We exploit the fact that the outer iterates correspond to the gradient ascent of a strongly-concave dual problem, therefore providing a linear convergence.
} 

\begin{theorem}\label{theorem:str_cvx_conv}
    Consider problem (\ref{problem}) under Assumptions \ref{assumption:safe_init}, \ref{assumption:smoothness}, \ref{def:beta}, and \ref{assumption:strong-conv}. Let the primal sub-problems of SCSA (Alg. \ref{alg:base}) be solved with accuracy 
    \textcolor{black}{
    $$\eta_{t+1} = \begin{cases}\frac{\mu_f \hat g(x_{t})^2}{128 L_g^2},
    &\text{ if } -\hat g(x_{t})\lambda_{t+1} > \e_c\\ 
    \min\{\frac{\mu_f}{M_L^2}\e_p^2, \frac{\e}{2}\}, &\text{ if } -\hat g(x_{t})\lambda_{t+1}\leq \e_c\end{cases},$$
    }
    %
     and constraint measurements be taken
    up to accuracy \textcolor{black}{$ \epsilon_t \leq \frac{-\hat g(x_{t-1})}{8}$ with probability $1 - T\de.$
    }
    Then, the algorithm stops after at most $T = O(\frac{\check\lambda}{\beta} + \frac{L_g^2 R^2}{\beta^2}\frac{M_f}{\mu_f}\log\frac{\check\lambda^2}{\mu_f\e_c})$ outer iterations of SCSA (Alg. \ref{alg:base})  with probability $1-\de$ 
    and outputs an $(\e_{p}, \e_c)$-approximate KKT point, 
    i.e.,
        $\|\nabla \mathcal L(x_T,\lambda_T) \|\leq \e_{p} 
        ,~
        \lambda_T(-g(x_T)) \leq \e_c.$
       Moreover, given  $\e_c \leq \e/2$, $\eta_T \leq \e/2$, we have 
    $ f(x_T) - f(x^*) \leq \e.$
    \end{theorem}
    \begin{proof}
        For the full proof see Appendix \ref{proof:theorem:str_cvx_conv}. Here we provide a short proof sketch. The outer steps of our algorithm 
$\lambda_{t+1} \leftarrow \max\{\lambda_t -\frac{(-\hat g(x_t))\mu_f}{\textcolor{black}{8}L_g^2},0\}$,
form a projected gradient ascent for the dual problem $d(\lambda)$ with a constant step size and approximate gradient. Moreover, $\forall t>0$: 
$\lambda_t$s are only decreasing: $ \lambda_{t+1} \leq \lambda_t$. Also, note that the dual problem $d$ is locally-strongly-concave at the region $\Omega_{\lambda}:=\{\lambda\in\R: g(x_{\lambda}) \geq \frac{-\beta}{2}\}$, thus implying fast convergence in that region. 

\emph{Firstly}, we show that the method's dual updates $\lambda_t$ can be outside of $\Omega_{\lambda}$ only for a finite number of steps $T_0$. Indeed, since outside of $\Omega_{\lambda}$ the gradient norm is lower bounded by a constant $\|\nabla d(\lambda)\| = - g(x_{\lambda}) \geq \beta/2$, we make at least a constant progress at each step, until reaching $\Omega_{\lambda}$  in  $T_0$ steps \emph{constant in $\e_c$}, bounded by  
$T_0 \leq
\frac{8L_g^2}{\mu_f}\frac{\check\lambda}{\beta}.$ If the method stops earlier by reaching $\lambda_t = 0$, 
then $T = T_0$.

\emph{Secondly}, inside this region $\Omega_{\lambda}$ we prove the linear convergence rate using standard techniques and careful disturbance analysis. 
In particular, \textcolor{black}{by using the bounds on $\epsilon_t$ and $\eta_t$ from the theorem conditions, we show that the estimate of the dual gradient $\|\hat \nabla d(\lambda_t) - \nabla d(\lambda_t)\| \leq \frac{\| \nabla d(\lambda_t)\|}{2}$ with probability $1-\de$ for all $t>0$ (using Boole's inequality), 
which is enough to have a linear convergence rate. That is, $(\lambda_{t+1} - \lambda_*) \leq \left(1 - \frac{\mu_f\mu_d}{16L_g^2}\right)^{t-T_0} \check\lambda$. We also show that $-\hat g(x_t)\lambda_{t+1} \leq \max\left\{\bar G, \frac{4L_g^2\check\lambda}{\mu_f}\right\}  (\lambda_{t+1} - \lambda_*)$, where $-\hat g(x_t) \leq \bar G$ for all $t$ (which holds for $\bar G = 2\hat g(\check x)$). Note that until stopping $-\hat g(x_t)\lambda_{t+1} >\e_c$, what implies that $\e_c \leq 
-\hat g(x_t)( \lambda_{t+1} - \lambda_*)\leq 
\max\left\{\bar G \check\lambda, \frac{4L_g^2\check\lambda^2}{\mu_f}\right\} \left(1 - \frac{\mu_f\mu_d}{16L_g^2}\right)^{t-T_0} .$
From the above we can directly upper bound the maximal number of steps before stopping:} 
\begin{align}\label{eq:T_str_convex}
T \leq
  & \frac{8L_g^2\check\lambda}{\beta\mu_f} +  \frac{16L_g^2}{\mu_d\mu_f}O\left(\ln\frac{\check\lambda^2}{\mu_f\e_c}\right).
\end{align} 
Using $\mu_d \geq \frac{\beta^2}{4R^2 M_L},$ we get the bound from the theorem.
Using the relation of primal optimality to the Lagrangian optimality, 
we show that given  $\e_c \leq \e/2$, $\eta_T \leq \e/2$, we have 
    $ f(x_T) - f(x^*) \leq \e.$
The expression for iteration complexity in Eq. (\ref{eq:T_str_convex}) gives us the \textit{linear convergence rate} for the dual iterations. 
    \end{proof}

Finally, we can guarantee the following sample complexity.
\begin{corollary}\label{corr:sample_complexity}
Consider stochastic oracle with gradient variance $\hat \sigma^2>0$ and value variance $\sigma^2>0.$ Then, the method requires 
    $ \mathcal N(\e) = \mathcal N_0(\check\eta) + O\left(\frac{L_g^2}{\beta^2}\frac{M_f}{\mu_f}\log \frac{1}{\e}\right) \left(\mathcal N_{\mathcal A}(\min \eta_t)
    + n_t(\epsilon_t)
    \right)$ 
    oracle calls to get $ f(x_T) - f(x^*) \leq \e$, where $\mathcal N_{\mathcal A}(\eta_t)$ is the sample complexity of the internal algorithm.
    For $\mathcal A \leftarrow \mathcal A_{SGD}$ and the conditions of $\check\eta, \eta_t, \epsilon_t, n_t$, the total sample complexity is
    $$\mathcal N(\e)
    =O\left(\frac{\check\lambda}{\beta} + \frac{M_f}{\mu_f}\frac{R^2}{\beta^2}\log\frac{1}{\e}\right) O\left(\frac{\sigma^2}{\e^2}+\frac{(L+\hat\sigma)^2}{\mu_f \e^2}\right).$$ 
\end{corollary}
For the proof see Appendix \ref{appendix:proof:corr:sample_complexity}.

\subsection{General Smooth and Convex Case}
Here, we consider the case where the objective and constraint are smooth and convex, but not necessarily strongly-convex. 
In this case, one can regularize the objective in order to make it strongly-convex.  
In particular, 
we can formalize the algorithm as follows:
\vspace{-2mm}
\begin{algorithm}[h]
\caption{SCSA (Alg. \ref{alg:base}) for convex problems}\label{alg:convex}

\begin{algorithmic}
\item[0.]{\emph{Initialization:}}  Feasible starting point $x_0$: $g(x_0) \leq -\beta < 0$, regularization parameter $\mu_f= \frac{\e}{R^2}>0$.
    \item[1.] Define $\hat{\mathcal P}: \min f(x) + \frac{\mu_f}{2}\|x - x_0\|^2 \text{ s.t. } g(x) \leq 0.$
    \item[2.] $ (x_{T}, \lambda_{T}) \leftarrow SCSA(\hat{\mathcal P}, x_{0})$ 
\end{algorithmic}
\end{algorithm}
\vspace{-2mm}
\begin{theorem}\label{thm:cvx}
    Consider problem  (\ref{problem})   under Assumptions \ref{assumption:safe_init}, \ref{assumption:smoothness}, \ref{def:beta}, and let $f(\cdot)$ and $g(\cdot)$ be convex. Let $x_0$ be a starting vector, 
    such that $\|x_0 -  x^*\| \leq R$, where $x^* = \arg\min_{x\in \mathcal X}f(x)$.
    Then, applying Algorithm \ref{alg:base} to a regularized problem $\min f(x) + \frac{\mu_f}{2}\|x - x_0\|^2 $ s.t. $g(x)\leq 0$ with $\mu_f = \e/R^2$ produces an output $x_T$ satisfying $f(x_T) - \min_{x\in \mathcal X} f(x) \leq \e$ in a total sample complexity of $\mathcal N(\e) = O(\frac{1}{\e^4})$, with probability $1-\de.$ Also, all iterates 
    are safe with probability $1-\de.$
\end{theorem}
For the proof see Appendix \ref{appendix:proof:thm:cvx}.


\section{Non-convex Case} \label{section:non-convex}
In this section, we address the setting of safe optimization of general non-convex and smooth objective and constraints; our performance measure is the an $(\e_p,\e_c)$-approximate KKT condition. 
First, let us make an additional regularity assumption, which is an extension of Mangasarian-Fromovitz Constraint Qualification (MFCQ):
\begin{assumption}\label{assumption:mfcq}
   There exists $\theta\in (0,\frac{\beta}{2}]$ such that for any point in $\{x\in \mathcal X|g(x)\geq -\theta\}$ there exists a constant $l>0$, such that  $\|\nabla g(x)\|\geq l.$
\end{assumption}
The classic MFCQ \citep{MANGASARIAN196737} is the regularity assumption on the constraints, guaranteeing that they have a uniform descent direction for all constraints at a local optimum. 
\Cref{assumption:mfcq}
guarantees $g(x)$ is decreasing towards the middle of the set  
at all points $\theta$-close to the boundary; 
it is the same as in \citet{usmanova2023log}.

  To address the non-convex case, first note that smoothness of $f(x)$ and $g(x)$ 
  implies their \emph{weak convexity}. That is, there exists such constant $\rho_f>0$ such that $f(x) + \frac{\rho_f}{2}\|x - x_0\|^2$ becomes convex, same with the constraint function.
    Then, we propose to use smoothing technique with the moving regularizer, inspired by 
    \citet{zhang2020singleloop}. 
    In particular, 
    at each step $k$ consider regularized problem:
    \begin{align}\label{sub-problem:nonconvex}
        \min_{x\in\R^d}~ & f(x) + \frac{\rho_f}{2}\|x - x_{k-1}\|^2\tag{$\mathcal P_{x_{k-1}}$},\\ 
        \text{ s.t. } &  g(x) + \frac{\rho_g}{2}\|x - x_{k-1}\|^2 \leq 0\nonumber
    \end{align} 
    with the corresponding Lagrangian $$\mathcal K_{x_{k-1}}(x, \lambda):= f(x) + \frac{\rho_f}{2}\|x - x_{k-1}\|^2 + \lambda \left(g(x) + \frac{\rho_g}{2}\|x - x_{k-1}\|^2\right)$$
    with regularizer parameters $\rho_f > M_f $ and $\rho_g > M_g$, so that both regularized objective and constraint functions are $\mu_i = \rho_i - M_i$-strongly convex and $\rho_i + M_i$-smooth for $i\in\{f,g\}$. 
    Note that $\mathcal K_{x_{k-1}}(\cdot, \lambda)$ is also $\mu_K = \rho_f - M_x + \lambda( \rho_g - M_g)$ strongly-convex, and $M_K = \rho_f + M_f + \lambda(\rho_g+M_g)$-smooth respectively. 
    
    We define the detailed updates in \Cref{alg:non-convex}. 
\textcolor{black}{Here, $\beta_k$ is a lower bound $\beta$ for problem $\mathcal P_{x_{k-1}}$, i.e.,   
 $ -\min_{x\in\R^d}\{g(x) + \frac{\rho_g}{2}\|x - x_{k-1}\|^2\} \geq \beta_k := -\hat g(x_{k-1})$, and $\text{SCSA}(\mathcal P_{x_{k-1}})$ means SCSA (Alg. \ref{alg:base})  applied to problem $\mathcal P_{x_{k-1}}$, and initialized directly with  primal-dual pair $(\check x_k, \check \lambda_k)$.}
    \begin{algorithm}
    [t]
    \caption{Safe Primal-Dual Method
    (SafePD)}\label{alg:non-convex}
    \begin{algorithmic}
    \item[0.]\emph{Preliminary step:}
    Set $\check \lambda_1 \leftarrow \frac{\Delta_f}{-\hat g(x_0)}$ 
    \FOR{$k\in\{1,\ldots,K\}$ }
    \item [1.] $\check x_{k} \leftarrow \check{\mathcal A}(\mathcal K_{x_{k-1}}(\cdot, \check \lambda_k),\R^d, x_{k-1})$ (to $ \check \eta_{k} \leq \frac{\mu_f \beta_k^2}{8L_g^2}$)
    \item 
    [2.] $(x_k,\lambda_k) \leftarrow \text{SCSA}(\mathcal P_{x_{k-1}}) 
    $, Alg. \ref{alg:base} init. at $(\check x_{k},\check \lambda_{k})$, and 
      solved up to accuracy $(\e_{p},\e_{c}) $
        \item 
        [3.]  \textbf{If} {$\|x_{k-1} - x_{k}\|\leq \min\left\{ \frac{\textcolor{black}{\e_{p}}}{\rho_f + \check\lambda_{k}\rho_g},\sqrt{\frac{2\e_{c}}{\check\lambda_{k}\rho_g}}\right\}$ } 
        \textbf{then} break
    \item[4.] \label{eq:lambda_warm}
        $\hat\lambda_{k+1} \leftarrow  2\lambda_{k} + \frac{\rho_f}{\rho_g} + \frac{\eta_k}{-g(x_{k})} + \frac{\hat \eta_{k+1}}{-g(x_{k})}$
    \ENDFOR
    \item{Return: $(x_K,\lambda_K)$}
    \end{algorithmic}
    \end{algorithm}
    \textcolor{black}{Note that 
    for sub-problem $\mathcal P_{x_{k-1}}$,
    we need $\check \lambda_k$ to be big enough that $g(\check x_k) + \frac{\rho_g}{2}\|\check x_k - x_{k-1}\|^2 \leq 0$, i.e.,  $x_{\check \lambda_k}$ is feasible subject to a regularized constraint. 
    It appears that we can set $\check \lambda_k$ according to Step \ref{eq:lambda_warm}  using the relation of the subsequent subproblems. This provides us with a less conservative warm start, compared to the original SCSA (Alg. \ref{alg:base}) initialization that would require 
    $\check \lambda_k \geq \frac{f(x_{k-1}) - \min_{x \in \mathcal X_k}f(x)}{-g(x_{k-1})}.$}
    Then, we can obtain the following convergence and sample complexity guarantees, while ensuring safety of all the updates.

    \begin{theorem}\label{theorem:noncvx_conv}
    Consider problem (\ref{problem}) under Assumptions \ref{assumption:safe_init}, \ref{assumption:smoothness}, \ref{def:beta} with non-convex objective and constraint, and let Assumption \ref{assumption:mfcq} hold for some constant $\theta>0$. Then, \Cref{alg:non-convex} stops after at most $K = O\left(\frac{1}{\e^2}
    \right)$ outer iterations reaching $(\e, \e)$-approximate KKT point with probability $1-\de K$. In total, it requires $\mathcal N = O\left(\frac{1}{\e^6\theta}\right)$ samples. Moreover, all iterates are feasible with probability $1-\de K$. 
    \end{theorem}
    The proof can be found in the Appendix \ref{proof:theorem:noncvx_conv}. 
    
\section{Extension to Multiple Safety Constraints}\label{section:extension}
Extending the method to multiple constraints $g_i(x) \leq 0 \forall i \in \{1,\ldots,m\}$ is  possible by replacing them with the max of all the constraints $g(x):= \max\{g_i(x)\}$ and applying any smoothing technique to the resulting function $g(x)$, such as randomized smoothing \citep{duchi2012randomizedsmoothingstochasticoptimization}. 
The convergence rate however worsens by a factor of $\frac{1}{\e}$ due to trade-off between the approximation quality and and smoothness constant. 
Indeed, the smoothed constraint $ g_{\nu}(x)$ with parameter $\nu>0$ is $O(\nu)$-approximation of the original constraint $g$ and it is $M_{\nu}$ smooth with $ M_{\nu} = O(\frac{1}{\nu}).$ \citep{duchi2012randomizedsmoothingstochasticoptimization}
    Therefore, we need $\nu = O(\e_c)$ to guarantee the required solution accuracy $- g(x)\lambda \leq \e_c$.
    Then, local strong concavity of the dual function $\mu_d = \frac{l^2}{M_f + \lambda M_{\nu}} = \Omega(\e_c)$. That is, the total complexity is increased to $T = \tilde O(\frac{1}{\e^3})$ for strongly convex case, and $N = \tilde O(\frac{1}{\e^5})$ for convex case, and $N = \tilde O(\frac{1}{\e^7} )$ for non-convex case. 
    Note that these bounds are still better than the previous bounds, for strongly-convex and convex cases, and similar to the baseline for non-convex case.

\section{Empirical Evaluations}
We compare LB-SGD \citep{usmanova2020safe} and SCSA (Alg. \ref{alg:base}) on a synthetic strongly-convex problem with black-box stochastic feedback. In particular, we consider 
   $ \min \|x - x_0\|^2,  
    \text{ s.t.} \|Ax - b\|^2-4\leq 0,$
with $A = [[I_{d-1},\bar{0}],[0,..,0,2]], b = [0,...,0,1]$ and $x_0 = [0,..,0,5].$ We consider noisy zeroth-order feedback with standard deviation $\sigma  =  0.01, 0.1$ and dimension $d = 2$. The shady area is the area between the max. and min. values over 10 runs.  The gradients are approximated by finite differences, similarly to \citet{usmanova2023log}. \textcolor{black}{Our experiments  demonstrate sensitivity of the LB-SGD to the increasing noise level, compared to our method. Note that when the noise is very small 
$ \sigma \in {0.01} $, the setting is very close to the noise-free setting, for which SGD used in our method is non-optimal. For larger $\sigma = 0.1$, the clear improvement can be seen. From the confidence-interval shades, one can also see SafePD is more stable. The complexity of SafePD can also be improved to $O\left(\log^2\frac{1}{\epsilon}\right)$ for the strongly-convex problems in the exact information case, since linear convergence can be achieved for the internal problems.}  
\begin{figure}
    \centering
\includegraphics[width=0.9\linewidth]{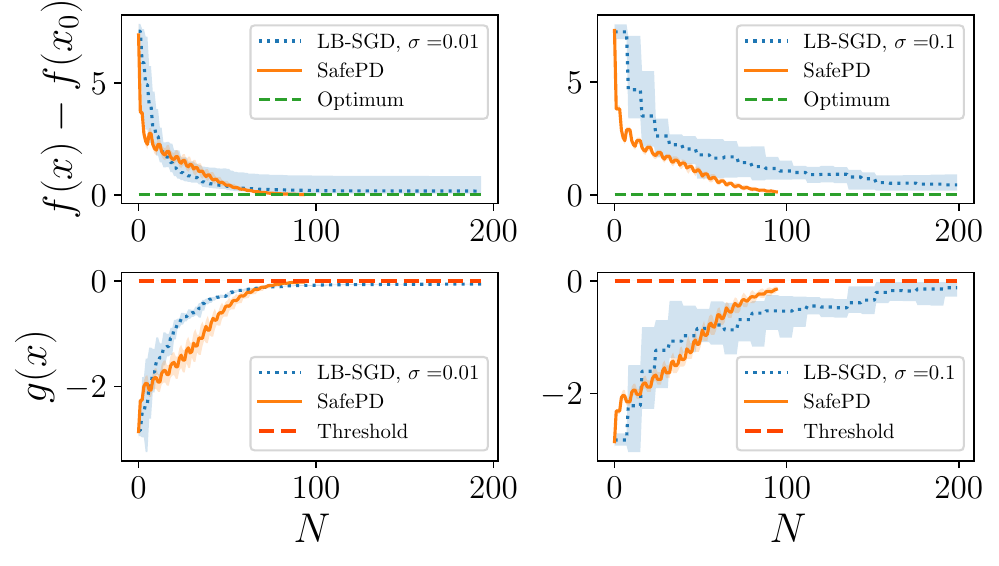}
    \caption{Comparison of SafePD (SCSA, Algorithm \ref{alg:base}), and LB-SGD. 
    }
    \label{fig:3}
\end{figure}
See Appendix \ref{experiments_non-convex} for non-convex experiments. Note that the experiments are only illustrative and the main focus of the paper is theoretical.


\section*{Acknowledgments and Disclosure of Funding} 
We thank Anton Rodomanov for a helpful discussion. This research was partially supported by Israel PBC- VATAT, by the Technion Artificial Intelligent Hub (Tech.AI), and by the Israel Science Foundation (grant No. 3109/24).




\bibliography{bibliography}

\clearpage
\appendix
\onecolumn
\section{Non-Convex Experiments}\label{experiments_non-convex}
Here, we demonstrate the performance of our method compared to LB-SGD on the following non-convex problem of minimising inverted Gaussian over a quadratic constraint.
\begin{align}
    \min_x & \exp^{-4 \|x\|^2_2}\\
    \text{s.t. }& 0.2 \|x - x_0\|^2_2 + 10. * (x[1] - x_0[1])^2 \leq r^2
\end{align} 
 with $r = 0.5$,  $x_0 = \frac{1}{\sqrt{d}} [1,\ldots,1],$ and $x[1]$ is the second component of vector $x$. 
 \begin{figure}[h]
    \centering
\includegraphics[width = \textwidth]{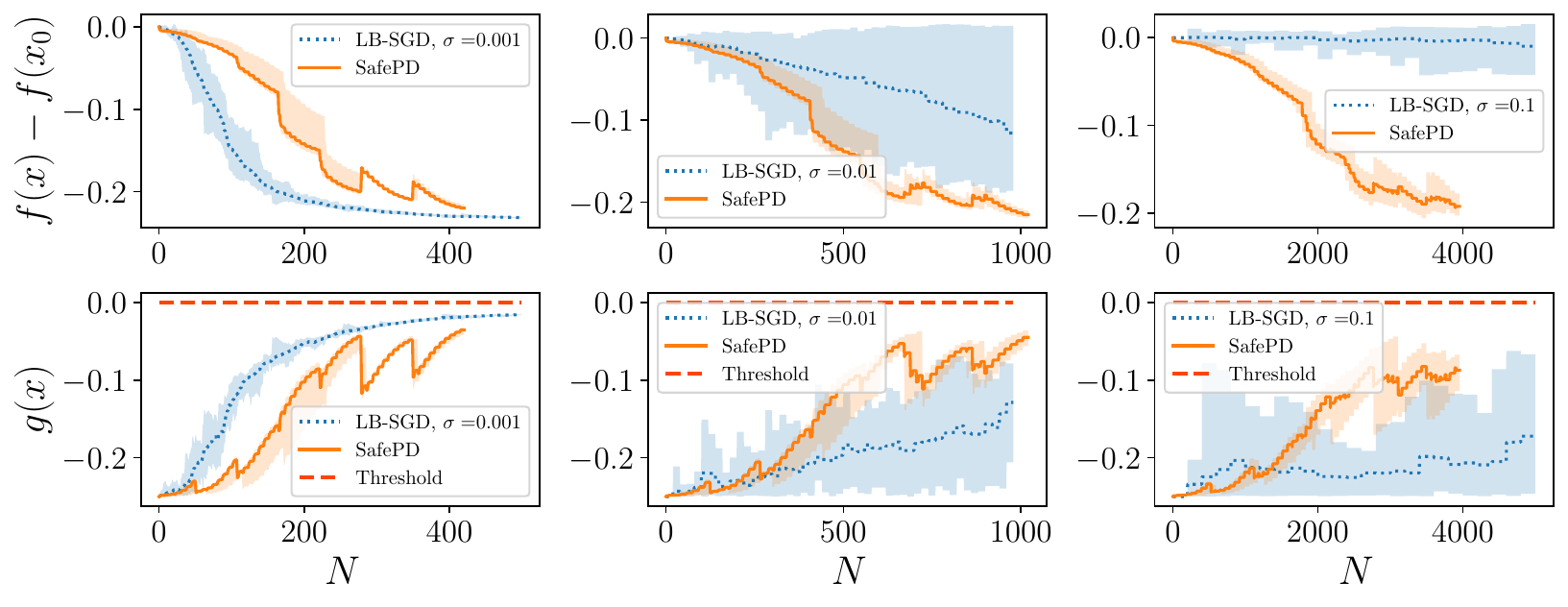}
 \caption{Comparison of SafePD (Alg. \ref{alg:non-convex}), and LB-SGD, on a non-convex problem. 
 }
    \label{fig:non-convex}
\end{figure}

 Here, we again observe that LB-SGD with higher noise gets less stable and slower in practice.
 \textcolor{black}{The goal of our experiments is to demonstrate sensitivity of the LB-SGD to increasing the noise level, compared to our method. Note that when the noise variance compared to the target accuracy is very small $ \sigma \in {0.001, 0.01} $, the given settings are very close to the noise-free setting. In the noise-free setting, the sample complexity of LB-SGD is $O\left(\frac{1}{\epsilon^2}\right)$ gradient steps for strongly-convex case, and $O\left(\frac{1}{\epsilon^3}\right)$ gradient steps for non-convex case, whereas for non-convex SafePD according to \Cref{corr:sample_complexity}, the sample complexity is still $O \left( \frac{1}{\epsilon^2} \right)$ for the strongly-convex case and $O \left( \frac{1}{\epsilon^6} \right)$ for the non-convex case, since in the experiments we use SGD for the internal optimization, which is non-optimal for noise-free setting.
However, note that the complexity of SafePD can be improved to $O\left(\log^2\frac{1}{\epsilon}\right)$ for the strongly-convex problems, $O\left(\frac{1}{\epsilon^4}\right)$ for non–convex problems in the exact information case, since SGD can be replaced then with a more efficient internal optimization algorithm achieving linear convergence rate in the exact information case. We do not use it in the experiments for consistency.} 

\textcolor{black}{
For the case of a higher noise level, we observe that $O\left(\frac{1}{\epsilon^6}\right)$ sample complexity of SafePD clearly outperforms $O \left( \frac{1}{\epsilon^7} \right)$ sample complexity of LB-SGD in practice, what confirms our theoretical findings. And it can also be seen from the confidence-interval shades, that our method is much more stable.}

\section{Additional Proofs}
\subsection{Proof of Lemma \ref{lemma:upper_lambda}}\label{proof:lemma:upper_lambda}
\begin{lemma} 
    There exists a feasible point $\tilde x\in \mathcal X$ such that $-g(\tilde x) \geq \beta > 0 ,$ and $\Delta_f:= \max_{x \in \mathcal X} \{f(x) - f(x^*)\}.$  Then, 
    $\lambda^* \leq \frac{\Delta_f}{\beta}.$
\end{lemma}  
\begin{proof}
    Let $(x^*,\lambda^*) = \min_{x\in\mathcal X}\max_{\lambda\geq 0}\mathcal L(x,\lambda).$ Then, for  feasible point $\tilde x$ we have 
    $f(\tilde x) - f(x^*) \geq \lambda^* (-g(\tilde x))$, that follows directly from $ \mathcal L(\tilde x,\lambda^*) \geq \mathcal L(x^*,\lambda^*).$ 
    It implies $\lambda^*$ is upper bounded by 
    $\lambda^* \leq \frac{f(\tilde x) - f(x^*)}{- g(\tilde x)} \leq \frac{\Delta_f}{\beta}.$
\end{proof}
\subsection{Proof of Lemma \ref{lemma:dual-smoothness}}\label{proof:lemma:dual-smoothness}
\begin{lemma}
    If $f(x)$ is $\mu_f$-strongly-convex, and $g(x)$ is convex and $L_g$-Lipschitz-continuous, then the dual function $d(\lambda)$ is $\frac{2 L_g^2}{\mu_f}$-smooth. 
\end{lemma}
\begin{proof}
    Since for every $\lambda>0$ the minimizer of $\min_{x\in \R^d}\La(x,\lambda)$ is unique $x_{\lambda}^*$, the dual function is differentiable with $\nabla d(\lambda) = g(x_{\lambda}^*).$  Next, due to the Lipschitz continuity, we have
    \begin{align*}
        &|g(x_{\lambda_1}^*) - g(x_{\lambda_2}^*)| \leq L_g\|x_{\lambda_1}^* - x_{\lambda_2}^*\|.\\
        &|g(x_{\lambda_1}^*) - g(x_{\lambda_2}^*)|^2 \leq L_g^2\|x_{\lambda_1}^* - x_{\lambda_2}^*\|^2
    \end{align*}
    Then, note that due to $\mu_f$-strong convexity of $\La(\cdot,\lambda_1)$, we have: 
    \begin{align}\label{eq:01}
        \frac{\mu_f}{2}\|x_{\lambda_1}^* - x_{\lambda_2}^*\|^2 & \leq \La(x_{\lambda_2}^*,\lambda_1) - \La(x_{\lambda_1}^*,\lambda_1) \nonumber\\
        & = \La(x_{\lambda_2}^*,\lambda_2) - \La(x_{\lambda_1}^*,\lambda_2) + (\lambda_1- \lambda_2)(g(x_{\lambda_2}^*) - g(x_{\lambda_1}^*)) \nonumber\\
        & \leq 0 + |\lambda_1- \lambda_2| |g(x_{\lambda_2}^*) - g(x_{\lambda_1}^*)|. 
    \end{align}
    Thus, 
     \begin{align*}
        &|g(x_{\lambda_1}^*) - g(x_{\lambda_2}^*)|^2 \leq L_g^2\|x_{\lambda_1}^* - x_{\lambda_2}^*\|^2 \leq \frac{2L_g^2}{\mu_f}|\lambda_1- \lambda_2| |g(x_{\lambda_1}^*) - g(x_{\lambda_2}^*)|,
    \end{align*}
    which directly implies
    \begin{align}\label{eq:dual_smoothness}
        &\|\nabla d(\lambda_1) - \nabla d(\lambda_2)\| \leq \frac{2L_g^2}{\mu_f}|\lambda_1- \lambda_2|.
    \end{align}
    \end{proof}

    \subsection{Proof of Lemma \ref{lemma:local_dual_strong_concavity}}\label{proof:lemma:local_dual_strong_concavity}
    \begin{lemma}
Problem (\ref{problem}) has a    \emph{$\mu_d$-locally strongly-concave} dual function with 
    \begin{align}
    \mu_d &= \frac{l^2}{M_f+\lambda M_g}
\end{align}
for all $\lambda$ such that  $\|\nabla g(x_{\lambda})\|\geq l>0$. For convex $g$, it implies
    $\mu_d \geq  \frac{\beta^2}{4R^2(M_f+\lambda M_g)}$
for all $\lambda$ such that  $g(x_{\lambda})\geq -\beta/2$, where $\beta := -\min_{x\in \mathcal X}\{g(x)\}$ if it exists, or a lower bound corresponding to some feasible $\tilde x$  such that  $-g(\tilde x) \geq \beta$.   
\end{lemma}
\begin{proof}
    By equation
(6.9), page 598 in \citep{bertsekas1999}, the Hessian of dual function is:
\begin{align}
    \nabla^2_{\lambda}d(\lambda) & = -\nabla g(x_{\lambda})^T(\nabla^2 f(x_{\lambda}) + \lambda \nabla^2 g(x_{\lambda}))^{-1}\nabla g(x_{\lambda})
\end{align}
Note that in our case of a single constraint, $d(\lambda)$ is a one-dimensional function, and the Hessian is just a scalar. Since $f$ is $\mu_f$-strongly-convex, i.e., $ \mu_f \preceq \nabla^2 f(x_{\lambda}) \preceq M_f$, and $g$ is convex,  i.e., $ \mu_f \preceq \nabla^2 f(x_{\lambda}) + \lambda \nabla^2 g(x_{\lambda}) \preceq M_f + \lambda M_g$, then  $\nabla^2_{\lambda}d(\lambda) \preceq -\frac{1}{M_f + \lambda M_g}\|\nabla g(x_{\lambda})\|^2.$ Hence, for all $\lambda$ such that $\|\nabla g(x_{\lambda})\| \geq l$, we get $\nabla^2_{\lambda}d(\lambda) \preceq-\frac{l^2}{M_f + \lambda M_g}.$ 

For convex $g(x),$ we can lower bound the norm of its gradient  on the set $g(x)\geq -\frac{\beta}{2}$. Indeed, for any point $x\in \mathcal X$ 
 and for convex constraint $g$ such that $g(x)\geq - \frac{\beta}{2}$, due to convexity we have  
    $
         g(x_{\lambda}) - g(\tilde x) \leq  \la\nabla g(x),   x - \tilde x\ra.$ Given the bounded diameter of the set   $\|\tilde x - x\|\leq R$, we get 
         $\|\nabla g(x_{\lambda})\|\geq \la\nabla g(x_{\lambda}),  
        \frac{x_{\lambda} - \tilde x}{\|x_{\lambda} - \tilde x\|}\ra \geq \frac{g(x_{\lambda}) - g(\tilde x)}{R} \geq \frac{\beta}{2R}$. 
    Therefore, $\nabla^2_{\lambda}d(\lambda) \preceq-\frac{\beta^2}{4R^2(M_f + \lambda M_g)}.$ \\
\end{proof}

\subsection{Proof of Lemma \ref{lemma:safe_init}.}\label{proof:lemma:safe_init}
\begin{proof}
    Let $\{x_t\}$ be a sequence for minimizing $\mathcal L(x,\check\lambda)$ such that $\mathcal L(x_t,\check\lambda) \leq\mathcal L(x_0,\check\lambda).$ That means:
    \begin{align*}
       f(x_t) + \check\lambda g(x_t)& \leq f(x_0) + \check\lambda g(x_0).\\
         g(x_t)& \leq \frac{f(x_0) - f(x_t)}{\check\lambda}  + g( x_0).
    \end{align*}
    To guarantee feasibility of all the iterates, i.e.,  $\forall t\geq T,  g(x_t) \leq 0$, it is enough to show $\frac{f( x_0) - f(x_t)}{\check \lambda}  + g( x_0) \leq 0$. 
    Thus, by choosing $\check\lambda = \frac{\Delta_f}{\alpha} \geq \frac{f(x_0) - \min_x f(x)}{-g(x_0)} ,$ we imply $g(x_t)\leq 0$ for all $x_t$ as long as $\mathcal L(x_t, \check\lambda) \leq \mathcal L( x_0,\check\lambda).$
\end{proof}

\subsection{Proof of Lemma \ref{lemma:dual_upd}}\label{proof:lemma:dual_upd}

 
\begin{proof}
Knowing that $\La(x_t,\lambda_t) - \La(x_{\lambda_t},\lambda_t) \leq \eta$, and using strong convexity of $\La(\cdot, \lambda_t)$ we can see $
    \frac{\mu_f}{2} \| x_{\lambda_t} -  x_{t}\|^2  \leq \eta_t
.$ Then, from triangle inequality, we get:
\begin{align}\label{eq:triangle}
    \| x_{\lambda_{t+1}} -  x_{t}\|& \leq \| x_{\lambda_{t+1}} -  x_{\lambda_t}\| + \| x_{\lambda_t} -  x_{t}\| 
    \leq \sqrt{\frac{\eta_t}{\mu_f}} + \| x_{\lambda_{t+1}} -  x_{\lambda_t}\|.
\end{align}
Recall that due to $\mu_f$-strong convexity of $\La(\cdot,\lambda_{t+1})$, similarly to \Cref{eq:01} in the proof of \Cref{lemma:dual-smoothness} we get: 
    \begin{align*}
        \frac{\mu_f}{2}\|x_{\lambda_{t+1}} - x_{\lambda_t}\|^2 
        & \leq \la \lambda_{t}- \lambda_{t+1},  g(x_{\lambda_{t+1}}) - g(x_{\lambda_t})\ra. 
    \end{align*}
Note that using the gradient direction $\lambda_{t+1}^i = \max\{\lambda_t^i + \gamma_t g_i(x_t), 0\}$, we always decrease (non-increase) the dual variables corresponding to the safety constraints since for our method we guarantee $g_i(x_t) \leq 0$ due to the safety assumption (and property that we prove by induction now). Therefore,  $\lambda_{t+1} \leq \lambda_t$. 
Hence,   
\begin{align*}
        \frac{\mu_f}{2}\|x_{\lambda_{t+1}} - x_{\lambda_t}\|^2 
        & \leq \la \lambda_{t}- \lambda_{t+1},  g(x_{\lambda_{t+1}}) - g(x_{\lambda_t})\ra 
        \leq L_g ( \lambda_{t}- \lambda_{t+1})\|x_{\lambda_{t+1}} - x_{\lambda_t}\|. 
    \end{align*}
    Thus, 
    \begin{align*}
        \frac{\mu_f}{2}\|x_{\lambda_{t+1}} - x_{\lambda_t}\| 
        & \leq (\lambda_{t}- \lambda_{t+1}) L_g. 
        \end{align*}
        Combining it with \cref{eq:triangle}
        \begin{align*}
        \|x_{\lambda_{t+1}} - x_{t}\|  \leq \|x_{\lambda_{t+1}} - x_{\lambda_t}\| + \sqrt{\frac{2\eta_t}{\mu_f}} &\leq  (\lambda_{t}- \lambda_{t+1}) \frac{2 L_g}{\mu_f} + \sqrt{\frac{2\eta_t}{\mu_f}} 
    \end{align*}
     This guarantees the safety of the next primal solution $\|x_{t} - x_{\lambda_{t+1}}\|\leq \frac{-\hat g(x_t)}{L_g}$
    if $\lambda_{t}- \lambda_{t+1} \leq  \frac{\mu_f}{4L^2_g} (- \hat g(x_t))$ and 
    $ \eta_t \leq  \mu_f\frac{(-\hat g(x_t))^2}{8L_g^2}.$
That is, solution of the problem $\min_{x\in \mathcal S(x_t)} \mathcal L(x,\lambda_{t+1})$ is equivalent to  $\min_{x\in \R^d} \mathcal L(x,\lambda_{t+1})$.  
Then, by using Algorithm $\mathcal A_{M,\mu}(\mathcal L(\cdot, \lambda_{t+1}), \mathcal S(x_t))$ with projections onto the safety set $ \mathcal S(x_t)$, we ensure that all the inner iterations are safe, including the last one $x_{t+1}$, and we can come arbitrarily close to $x_{\lambda_{t+1}}$ in a safe way. 
\end{proof}

\section{Convergence: Proof of Theorem \ref{theorem:str_cvx_conv}} 
\label{proof:theorem:str_cvx_conv} 

First, we relate the primal optimality to the Lagrangian optimality. 
\begin{lemma}\label{lemma:relation}
    Let  $\mathcal L(x_T,\lambda_T) - \mathcal L(x_{\lambda_T},\lambda_T) \leq \eta_T$, and the approximate complementarity slackness be bounded by $-\lambda_T g(x_T)\leq \e_c$. If  
    $\eta_T\leq \e/2 $ and $ \e_c \leq \e/2$, then $x_T$ satisfies the $\e$-primal optimality 
    $f(x_T) - f(x^*) \leq \e.$
\end{lemma}
\begin{proof} 
Note that $\mathcal L(x_T,\lambda_T) - \mathcal L(x^*,\lambda^*)
= \mathcal L(x_T,\lambda_T) - \mathcal L(x_{\lambda_T},\lambda_T) + \underbrace{\mathcal L(x_{\lambda_T},\lambda_T)}_{d(\lambda_T)} - \underbrace{\mathcal L(x^*,\lambda^*)}_{d(\lambda^*)}  \leq \eta_T + d(\lambda_T) - d(\lambda^*)\leq \eta_T,
$
where the last inequality is since $d(\lambda^*) = \max_{\lambda\geq 0}d(\lambda) \geq d(\lambda_T)$. 
     Thus,  
     $f(x_T) - f(x^*)  = \mathcal L(x_T,\lambda_T) - \mathcal L(x^*,\lambda^*) - \lambda_T g(x_T)+0 \leq \eta_T + \e_c \leq \e,$ by condition of the Lemma.
\end{proof}
Below, we provide the full proof of \Cref{theorem:str_cvx_conv}.
\begin{proof}
  Consider the process 
$\lambda_{t+1} \leftarrow \max\{\lambda_t -\frac{(-\hat g(x_t))\mu_f}{4L_g^2},0\}$,
starting from $\lambda_1 = \check\lambda.$ 

By $T_0$ we denote the number of iterations after which $g(x^*_{\lambda_t}) \geq -\frac{\beta}{2}$ or $\lambda_t = 0.$ Such $T_0>0$ exists since  $\lambda_t$ is a decreasing sequence for feasible updates with $\hat g(x_t)<0$, and $\nabla d(\lambda) = g(x^*_{\lambda})$ is monotonously increasing with $\lambda \rightarrow 0$. 


\emph{Firstly,} we can upper bound  $T_0$ as follows.  Note that $ 0\leq \lambda_{T_0} \leq \lambda_0$ since we only decrease $\lambda$. Recall that until reaching $\lambda_{t+1} = 0$, by definition of the algorithm we have 
$\lambda_t - \lambda_{t+1} = \frac{-\hat g(x_t) \mu_f}{4L_g^2}.$
Hence, 
$$\lambda_{0} - \lambda_{T_0} =  \sum_{t=0}^{T_0-1} (\lambda_t - \lambda_{t+1}) = \sum_{t=0}^{T_0-1} \frac{-\hat g(x_t) \mu_f}{4L_g^2} \geq T_0 \frac{\beta \mu_f}{8L_g^2},$$
for all $t\leq T_0$ while $g(x_t) <- \frac{\beta}{2}$. The above directly leads to: $T_0 \leq \frac{8L_g^2(\lambda_0 - \lambda_{T_0})}{\beta\mu_f} \leq \frac{8L_g^2 \lambda_0}{\beta\mu_f} .$

\emph{Secondly,} below we upper bound the number of steps required for convergence after reaching $T_0$.  We have strong convexity of the Lagrangian which gives us $\frac{\mu_f}{2}\|x_t - x_\lambda\|^2 \leq  \mathcal L(x_t) -  \mathcal L(x^*_{\lambda_t}) \leq \eta_t$, and Lipschitz continuity of the constraint: $g(x_t) - g(x_{\lambda_t}) \leq L_g\|x_t - x_{\lambda_t} \|$. Combining these together, we get with probability $\geq 1-\de$, until converging \textcolor{black}{(while $-\hat g(x_{t})\lambda_{t} > -\hat g(x_{t})\lambda_{t+1} > \e_c$)}: 
\begin{align}\label{eq:3.10.2}
\lambda_{t+1} - \lambda^*
& = \lambda_t - \lambda^* - \frac{(-\hat g(x_t))\mu_f}{4L_g^2}\leq \lambda_t - \lambda^* - \frac{(- g(x_t) - \epsilon_t )\mu_f}{4L_g^2}  \nonumber\\
& \leq \lambda_t - \lambda^* - \frac{(- g(x_{\lambda_t}) - \epsilon_t  - L_g\sqrt{2\eta_t/\mu_f})\mu_f}{4L_g^2},
\end{align}
where the second inequality is due to strong convexity $\frac{\mu_f}{2}\|x_{t} - x_{\lambda_t}\|^2\leq \eta_t,$ and $L_g$-Lipschitz continuity of $g$.
\begin{lemma}
        \textcolor{black}{If $\epsilon_t \leq \frac{-\hat g(x_{t-1})}{8}$ and $ \eta_{t} 
        = \frac{\mu_f \hat g(x_{t-1})^2}{128 L_g^2}$, then,  $\frac{- g(x_{\lambda_t})}{2} \leq -\hat g(x_{t}) \leq - 2g(x_{\lambda_t}) .$}
\end{lemma}
\begin{proof}
       \textcolor{black}{ Note that $\epsilon_t \leq \frac{-\hat g(x_{t-1})}{8} \leq \frac{-g(x_{\lambda_t})}{4}$(by Lemma \ref{lemma:dual_upd}). Similarly, using 
        $\eta_{t} 
        = \frac{\mu_f \hat g(x_{t-1})^2}{128 L_g^2}$, we get
        $ L_g\sqrt{\frac{2\eta_t}{\mu_f}} \leq  \frac{-g(x_{\lambda_t})}{4}.$ }
\end{proof}

Then, combined with Eq. (\ref{eq:3.10.2}), it implies 

$$\lambda_{t+1} - \lambda^*
 \leq \lambda_t - \lambda^* - \frac{- g(x_{\lambda_t})\mu_f}{8L_g^2}.$$ 
 Note that after $T_0$ we have $\mu_d$-strong concavity of the dual function (or stop because of reaching $0$, in this case $T = T_0$.)
 The strong concavity implies: (For the proof see Appendix \ref{proof:lemma:lambda_str_bound}).
\begin{lemma}\label{lemma:lambda_str_bound}
    For any $\lambda\geq 0$ such that $ g(x_{\lambda})\geq-\beta/2$ we have \begin{align} -g(x_{\lambda}) &\geq \frac{\mu_d}{2} (\lambda - \lambda^*). 
    \label{eq:dual-local-str-cvx} \end{align}
\end{lemma}
 By using Lemma \ref{lemma:lambda_str_bound} (\ref{eq:dual-local-str-cvx}), we obtain:
\begin{align*}
\lambda_{t+1} - \lambda^* 
& \leq \lambda_t - \lambda^* - \frac{ (\lambda_t - \lambda^*) \mu_f\mu_d}{16L_g^2} 
= (\lambda_t - \lambda^*)\left(1 
- \frac{\mu_f \mu_d}{16L_g^2}\right).
\end{align*}
By continuing inductively until $\lambda_{T_0}$, for all $t \geq T_0$ we get:
\begin{align*}
\lambda_{t+1} - \lambda^* &\leq \left(1 - \frac{\mu_f\mu_d}{16L_g^2}\right)^{t-T_0}(\lambda_{T_0} - \lambda^*) \leq \left(1 - \frac{\mu_f\mu_d}{16L_g^2}\right)^{t-T_0}\check\lambda .
\end{align*} 

\emph{Thirdly,} 
Note that the algorithm stops as soon as $-\hat g(x_{t})\lambda_{t+1} \leq \e_c,$ in this case it founds $x_{t+1}$ with accuracy $\eta_T = O(\e_p^2)$, which leads to $(x_{t+1}, \lambda_{t+1})$ being an $(\e_c, \e_p)$-approximate KKT point. 

Until then, $\e_c \leq -\hat g(x_t)\lambda_{t+1} = -\hat g(x_t)(\lambda_{t+1} -  \lambda^*) + (- \hat g(x_t))\lambda^*.$  From here, we consider 2 cases. 

\textbf{Case 1.} If $\lambda^*>0$, then $\nabla d(\lambda^*) = g(x^*) = 0$. Using that:
$$ 
- \hat g(x_{t}) \leq -g(x_{\lambda_{t}}) + \e_t + L_g\sqrt{\frac{\eta_{t}}{\mu_f}} \leq - 2 g(x_{\lambda_{t+1}}) - \frac{-\hat g(x_{t})}{2},$$
we get:
$$- \hat g(x_{t}) =  4(- g(x_{\lambda_{t+1}}))= 4\|\nabla d(\lambda_{t+1})\| \leq \frac{4L_g^2}{\mu_f}(\lambda_{t+1} - \lambda_*).$$
Then $$ -\hat g(x_t)\lambda_{t+1} \leq \frac{4L_g^2}{\mu_f}\lambda_{t+1} (\lambda_{t+1} - \lambda_*)$$

\textbf{Case 2.}
If $\lambda^* = 0,$ then 
$$  -\hat g(x_t)\lambda_{t+1} = -\hat g(x_t)(\lambda_{t+1}-\lambda^*) \leq G (\lambda_{t+1}-\lambda^*)$$

That is, until stopping:
$$ \e_c \leq -\hat g(x_t)\lambda_{t+1} \leq \max\left\{G, \frac{4L_g^2\check\lambda}{\mu_f}\right\}  (\lambda_{t+1} - \lambda_*)\leq  \max\left\{G, \frac{4L_g^2\check\lambda}{\mu_f}\right\} \left(1 - \frac{\mu_f\mu_d}{16L_g^2}\right)^{t-T_0} \check\lambda.$$
 What implies that until stopping we have at most 

$$t-T_0\leq\frac{\ln\left(\max\left\{\frac{4L_g^2\check\lambda^2}{\mu_f}, \bar G\check\lambda\right\}\frac{1}{\e_c}\right) }{ -\ln (1 - \frac{\mu_d\mu_f}{16L_g^2})}.$$
Note that $\frac{z}{1+z} \leq \ln(1+z) \forall z >-1$, that is $\ln(1 - \frac{\mu_d\mu_f}{16L_g^2}) \geq \frac{- \mu_g \mu_f}{16L_g^2}.$ Thus, we require
\begin{align}
    t\leq T_0 + \frac{16L_g^2}{\mu_d\mu_f} \ln\left(\max\left\{\frac{4L_g^2\check\lambda^2}{\mu_f}, \bar G\check\lambda\right\}\frac{1}{\e_c}\right) .
\end{align}

\emph{Then,} in total we require $T>0$ steps with 
\begin{align}
T \leq \frac{8L_g^2\check\lambda}{\beta\mu_f} +  \frac{16L_g^2}{\mu_d\mu_f}O\left(\ln\frac{\check\lambda^2}{\mu_f\e_c}\right) .
\end{align} Recall that $\mu_d \geq \frac{\beta^2}{4R^2 M_L},$ then we get the bound from the theorem.

\emph{Finally,} recall that for the last step when $ -\hat g(x_{T-1})\lambda_{T} \leq \e_c$ we have $\eta_{T} \leq \frac{\mu_f\e_p^2}{(M_f + \lambda_T M_g)^2}$. 
\textcolor{black}{From smoothness of Lagrangian we have $\|\nabla \mathcal L(x_T)\| \leq (M_f + \lambda_T M_g)\|x_T - x_{\lambda_T}\| \leq (M_f + \lambda_T M_g)\sqrt{\frac{\mathcal L(x_T) - \mathcal L(x_{\lambda_T})}{\mu_f} } \leq (M_f + \lambda_T M_g)\sqrt{\frac{\eta_T}{\mu_f}} \leq \e_p.$} 
 What directly implies $\|\nabla \mathcal L(x_T)\| \leq \e_p $, i.e., $(x_{t+1},\lambda_{t+1}) $ is a $(\e_p, \e_c)$-KKT point.

Then, by using Lemma \ref{lemma:relation}, if $\eta_t \leq \e/2$ and $\e_c \leq \e/2$, we get an $\e$-accurate solution  $|f(x_T) - f(x^*)| \leq \e$. The expression for iteration complexity in Eq. (\ref{eq:T_str_convex}) gives us the \textit{linear convergence rate} for the dual iterations. 
\end{proof}

\subsection{Proof of Lemma \ref{lemma:lambda_str_bound}}\label{proof:lemma:lambda_str_bound}
\begin{lemma}
    For any $\lambda\geq 0$ such that $ g(x_{\lambda})\geq-\beta/2$ we have \begin{align*} -g(x_{\lambda}) &\geq \frac{\mu_d}{2} (\lambda - \lambda^*).  \end{align*}
   \textcolor{black}{ where $\lambda^* = \arg\min_{\lambda \geq 0} d(\lambda)$}
\end{lemma}
\begin{proof}
Recall that from \Cref{lemma:local_dual_strong_concavity}, we have for all $\lambda\geq 0$ such that $ g(x_{\lambda})\geq-\beta/2$ the local strong concavity:
$d(\hat \lambda^*) \geq d(\lambda) + \frac{\mu_d}{2}\|\lambda - \hat \lambda^*\|^2,$ \textcolor{black}{where $\hat\lambda^* = \arg\min_{\lambda\in\R} d(\lambda).$}
By using concavity, for all $\lambda>\lambda^*$ and with $g(x_{\lambda})\leq 0$ we get:
\begin{align*}
     \la \nabla_{\lambda} d(\lambda), \hat \lambda^* - \lambda \ra &\geq 
    d(\hat\lambda^*) - d(\lambda) \geq  \frac{\mu_d}{2}\|\lambda - \hat\lambda^*\|^2.\nonumber \\
    -g(x_{\lambda}) | \lambda -\hat\lambda^*|  &=  -g(x_{\lambda}) ( \lambda -\hat\lambda^*)  \geq  \frac{\mu_d}{2}|\lambda - \hat\lambda^*|^2\nonumber \\
    -g(x_{\lambda}) &\geq \frac{\mu_d}{2} |\lambda - \hat \lambda^*| = \frac{\mu_d}{2} (\lambda - \hat \lambda^*) \geq \frac{\mu_d}{2} (\lambda - \lambda^*), 
\end{align*}
since $\lambda^* = \max\{\hat \lambda^*, 0\}$
\end{proof}

\subsection{Proof of Theorem \ref{thm:cvx}}\label{appendix:proof:thm:cvx}
\begin{proof}
   Note that the new objective of problem $\hat{\mathcal P}$ is $\e$-strongly-convex, which allows to apply the guarantees of the previous section, implying that 
    to reach the accuracy $\e/2$ of the regularized problem the algorithm requires   
    $ \mathcal N = O(\frac{\Lambda}{\beta} + \frac{R^2}{\mu_f\beta^2}\log\frac{1}{\e})O\left(\frac{\sigma^2}{\e^2}+\frac{(L + \sigma)^2}{\mu_f\e_{p}^2}\right) = \tilde O(\frac{1}{\e\beta}\left(\frac{1}{\e^2} + \frac{1}{\e^3}\right)) = O(\frac{1}{\e^4}).$ 
    Then, the algorithm outputs $x_T$ satisfying: 
  $ f(x_T) + \frac{\mu_f}{2}\|x_0 - x_T\|^2  - f(x^*) - \frac{\mu_f}{2}\|x_0 -  x^*\|^2  \leq \frac{\e}{2},$
    which is due to $x^*$ is the minimizer, and by the property of the algorithm reaching accuracy $\frac{\e}{2}.$ 
    The above implies the target accuracy $f( x_T)   -  f(x^*) \leq $
        $\frac{\e}{2} + \frac{\mu_f}{2}\|x_0 - x^*\|^2 -  \frac{\mu_f}{2}\|x_0 - x_T\|^2\leq \frac{\e}{2} + \frac{\mu_f}{2}R^2 \leq \e.$
    Safety of this method follows directly from the safety of the SCSA (Alg. \ref{alg:base}) algorithm.
\end{proof}

\subsection{Proof of Corollary \ref{corr:sample_complexity}}\label{appendix:proof:corr:sample_complexity}
\begin{proof}
For the internal primal iterations we can use any algorithm for strongly-convex smooth stochastic problems (e.g. projected SGD), and obtain the complexity $O(\frac{1}{\eta_t}) = O(\frac{1}{\e^2})$. 
    Recall that $\epsilon_t = \frac{- \hat g(x_t)}{8},$ that implies, before stopping $\e_t\geq \frac{\e}{8 \lambda_t}$ (if $\lambda_t >0$), or, if $\lambda_t = 0$, measuring the value of $g(x_t)$ is not needed any more, since the solution is strictly in the interior. 
    In order to have $\epsilon_t$-accurate measurements of $g(x_t)$ we need $n_t = O(\frac{\sigma^2}{\epsilon_t^2}) = O(\frac{\sigma^2}{\e^2})$ extra measurements at each of outer iterations. That is, in total we need 
     $\mathcal N(\e) = \mathcal N_0(\check\eta) + \sum_{t=1}^T (\mathcal N_{\mathcal A}(\eta_t) + n_t)$ 
    first- and zeroth-order measurements in total. If we use descent-SGD with minibatches for the initialization algorithm $\check{\mathcal A}$, 
    from Lemma \ref{proof:lemma:init_sgd}) we get its sample complexity bound  
$\mathcal N_0(\check \eta) = \tilde O(\frac{1}{\check\eta^2}\log \frac{1}{\check\eta}) = \tilde O(\frac{1}{\alpha^2}\log \frac{1}{\alpha}).$
\end{proof}

\section{Non-convex Case: Proof of Theorem  \ref{theorem:noncvx_conv}}\label{proof:theorem:noncvx_conv}
 \paragraph{Notations} For the analysis, we construct the following potential function 
    $\phi_k := \mathcal K_{x_{k-1}}(x_k, \lambda_k) $ for $k\geq 1$.  
    At every iteration, we solve regularized subproblem \ref{sub-problem:nonconvex} up to accuracy $(\e_{p}, \e_{c})$, that is:
      \begin{align*}
      \mathcal K_{x_{k-1}}(x_k,\lambda_k) - \mathcal K_{x_{k-1}}( x_k^*,\lambda_k)& \leq \eta_k =\frac{\mu_f}{2}\e_{p}^2 \\ 
       \|\nabla_x \mathcal K_{x_{k-1}}(x_k,\lambda_k)\| & \leq \e_{p}\\
          - \lambda_k\left( g(x_k) + \frac{\rho_g}{2}\|x_k - x_{k-1}\|^2\right) & \leq  \e_{c}, 
      \end{align*}
       Indeed, $\|\nabla_x \mathcal K_{x_{k-1}}(x_k,\lambda_k)\|^2 \leq  M_K( \mathcal K_{x_{k-1}}(x_k,\lambda_k) - \mathcal K_{x_{k-1}}(\hat x_k,\lambda_k)) \leq \e_{p}^2$. In the above, $\hat x_k = \arg\min_x \mathcal K_{x_{k-1}}(\cdot,\lambda_k).$ 
       
\textcolor{black}{\paragraph{Stopping criterion} 
Below, see the key lemma for the convergence of our approach, that also provides us with the stopping criterion.
\begin{lemma}\label{lemma:KKT_con_cvx}
As soon as $\|x_k - x_{k-1}\|\leq \min\left\{ \frac{\e_{p}}{\rho_f + \check\lambda_{k}\rho_g},\sqrt{\frac{2\e_{c}}{\check\lambda_{k}\rho_g}}\right\}$, the output 
every outer iteration $(x_k,\lambda_k)$ satisfies $(\e_{p} + \e_p ,\e_{c}+\e_c)$-KKT condition for the original problem. 
\end{lemma}
\begin{proof}
Indeed, $\|\nabla \mathcal L(x_k,\lambda_k)\| \leq$  $\|\nabla \mathcal K_{x_{k-1}}(x_k,\lambda_k)\|  +  (\rho_f + \check\lambda_{k}\rho_g)\|x_k - x_{k-1}\| \leq \e_{p} + \e_{p}.$
    Also, we get: 
    $ - \lambda_k\left( g(x_k) + \frac{\rho_g}{2}\|x_k - x_{k-1}\|^2\right) \leq \e_{c} $ which implies $-\lambda_k g(x_k) \leq \e_{c} + \lambda_k\frac{\rho_g}{2}\|x_k - x_{k-1}\|^2 \leq \e_{c} + \e_c.$ 
\end{proof}}

\begin{theorem}
    Consider problem (\ref{problem}) with non-convex objective and constraints, and let Assumption \ref{assumption:mfcq} hold for some constant $\theta>0$. Then, \Cref{alg:non-convex} stops after at most $K = O\left(\frac{1}{\e_{p}^2}\right)$ outer iterations. In total, it requires $\mathcal N = O\left(\frac{1}{\e_{p}^6\theta}\right)$ measurements.
    \end{theorem}
  \begin{proof}
       First, note that $\phi_k$ is lower bounded. Second, let us bound the improvement per iteration. \\
        \textit{Bounding an improvement per iteration.}
 From the definition: 
    \begin{align}
         \phi_{k+1} - \phi_k &= \mathcal K_{x_k}(x_{k+1}, \lambda_{k+1}) - \mathcal K_{x_{k-1}}(x_{k}, \lambda_{k}) \nonumber \\
         & =  \mathcal K_{x_k}(x_{k+1}, \lambda_{k+1}) - \mathcal K_{x_k}(x_{k}, \lambda_{k+1}) \label{eq:pot_1} \\
        & + \mathcal K_{x_k}(x_{k}, \lambda_{k+1}) - \mathcal K_{x_{k-1}}(x_{k}, \lambda_{k+1}) \label{eq:pot_2}\\
        & +  \mathcal K_{x_{k-1}}(x_{k}, \lambda_{k+1}) - \mathcal K_{x_{k-1}}(x_{k}, \lambda_{k}).\label{eq:pot_3}
    \end{align}
     Then, since $x_{k+1}$ is an 
     $\e_{k+1}$-approximate minimizer of $\mathcal K_{x_k}(\cdot,\lambda_{k+1})$, and using strong-convexity, \Cref{eq:pot_1} we bound by 
     \begin{align*}
        \mathcal K_{x_k}(x_{k+1}, \lambda_{k+1}) - \mathcal K_{x_k}(x_{k}, \lambda_{k+1}) 
         & \leq -\la\nabla\mathcal K_{x_k}(x_{k+1}, \lambda_{k+1}), x_{k+1} -x_k\ra - \frac{\mu_K}{2}\|x_k - x_{k+1}\|^2\\
         & \leq \|\nabla\mathcal K_{x_k}(x_{k+1}, \lambda_{k+1})\|\|x_k - x_{k+1}\| - \frac{\mu_K}{2}\|x_k - x_{k+1}\|^2\\
         & \leq \e_{p,k+1}\|x_k - x_{k+1}\| - \frac{\mu_K}{2}\|x_k - x_{k+1}\|^2
         . 
     \end{align*} 
     
     For \Cref{eq:pot_3}, 
     from linearity of $\mathcal K_{x_{k-1}}(x_k, \cdot)$ on $\lambda$, $(\e_{p},\e_{c})$-KKT condition on $\mathcal K_{x_{k-1}}$, $\lambda_{k+1}\geq 0$ and $ g(x_{k})\leq 0$ we have:
     \begin{align*}
          \mathcal K_{x_{k-1}}(x_{k}, \lambda_{k+1}) - \mathcal K_{x_{k-1}}(x_{k}, \lambda_{k}) 
         &= ( \lambda_{k+1} - \lambda_{k})  (g(x_k)+ \frac{\rho_g}{2}\|x_k - x_{k-1}\|^2)\\ 
         & = \lambda_{k+1}  g(x_k) + \lambda_{k+1}\frac{\rho_g}{2}\|x_{k} - x_{k-1}\|^2 +  \underbrace{\lambda_k(-g(x_{k})) - \lambda_k\frac{\rho_g}{2}\|x_{k} - x_{k-1}\|^2}_{\leq  \e_{c}} ) \\
         & \leq - \lambda_{k+1}( -g(x_k)  -\frac{\rho_g}{2}\|x_{k} - x_{k-1}\|^2) +  \e_{c}. 
     \end{align*} 
     And finally, for \Cref{eq:pot_2} we have
    \begin{align*}
         \mathcal K_{x_k}(x_{k}, \lambda_{k+1}) - \mathcal K_{x_{k-1}}(x_{k}, \lambda_{k+1}) 
         &= \left(\frac{\rho_f}{2}+\lambda_{k+1}\frac{\rho_g}{2}\right)\|x_k - x_k\|^2 - \left(\frac{\rho_f}{2}+\lambda_{k+1}\frac{\rho_g}{2}\right)\|x_k - x_{k-1}\|^2\\ 
         & = -\frac{\rho_f + \lambda_{k+1} \rho_g}{2}  \|x_k - x_{k-1}\|^2.
    \end{align*}
    Therefore, we have 
    \begin{align*}
     \phi_{k+1} - \phi_k  &\leq \e_{p,k+1}\|x_k - x_{k+1}\| - \frac{\mu_K}{2}\|x_k - x_{k+1}\|^2 - \lambda_{k+1}( -g(x_k)  -\frac{\rho_g}{2}\|x_{k} - x_{k-1}\|^2)+ \e_{c} -\frac{\rho_f + \lambda_{k+1} \rho_g}{2}  \|x_k - x_{k-1}\|^2
   \\
    & = (\e_{c} + \e_{p,k+1}\|x_{k+1} - x_{k}\|) - \frac{\rho_f}{2}  \|x_k - x_{k-1}\|^2 \\
    & - \frac{\mu_K}{2}\|x_k - x_{k+1}\|^2 - \lambda_{k+1}(-g(x_k))
    .
    \end{align*}
     \textit{Bounding number of outer steps.} Summing up over $k$, we get
    \begin{align*}
    & \phi_0 - \phi^*  \geq \phi_0 - \phi_K = \sum_{k=1}^K \phi_{k} - \phi_{k+1} \\
    & \geq \sum_{k=1}^K-(\e_{c} + \e_{p,k+1}\|x_{k+1} - x_{k}\|) + \frac{\rho_f}{2}  \|x_k - x_{k-1}\|^2 + \frac{\mu_K}{2}\|x_k - x_{k+1}\|^2 + \lambda_{k+1}(-g(x_k))\\
    & \geq -\sum_{k=1}^K(\e_{c} + \e_{p}\|x_{k} - x_{k-1}\|) + \sum_{k=1}^K\frac{\rho_f+\mu_K}{2}  \|x_k - x_{k-1}\|^2 + \lambda_{k+1}(-g(x_k))\\
    & \geq -\sum_{k=1}^K(\e_{c} + \e_{p}\|x_{k} - x_{k-1}\|) + \sum_{k=1}^K\frac{\rho_f+\mu_K +\lambda_{k+1}\rho_g}{2}  \|x_k - x_{k-1}\|^2\\
    & =  \sum_{k=1}^K\left(\underbrace{\left(\underbrace{\frac{\rho_f + \mu_K +\lambda_{k+1}\rho_g}{2}  \|x_k - x_{k-1}\| - \e_{p}}_{(1)}\right)\|x_k - x_{k-1}\| -  \e_{c}}_{(2)}\right).
    \end{align*}
    The last inequality is by safety of the updates $g(x_{k})+\frac{\rho_g}{2}\|x_{k} - x_{k-1}\|^2\leq 0$, i.e., 
        $-g(x_{k}) \geq \frac{\rho_g}{2}\|x_{k} - x_{k-1}\|^2.$
    Then, $(1) \geq \frac{\rho_f+\mu_K +\lambda_{k+1}\rho_g}{4} \|x_k - x_{k-1}\| $ if 
    \begin{align}\label{eq:e_pk}
    \e_{p} \leq \frac{\rho_f+\mu_K +\lambda_{k+1}\rho_g}{4} \|x_k - x_{k-1}\|.
    \end{align}
    
    Also, $(2) \geq \frac{\rho_f+\mu_K +\lambda_{k+1}\rho_g}{8} \|x_k - x_{k-1}\|^2 $ if 
    \begin{align}\label{eq:e_k}
    \e_{c} \leq \frac{\rho_f+\mu_K +\lambda_{k+1}\rho_g}{4} \|x_k - x_{k-1}\|^2.
    \end{align}
    
    In this case, we have:
    \begin{align*}
      \phi_0 - \phi^* & \geq  \sum_{k=1}^K\frac{\rho_f+\mu_K +\lambda_{k+1}\rho_g}{2}  \|x_k - x_{k-1}\|^2.
    \end{align*}
    Recall that before stopping,
    $\|x_k - x_{k-1}\| \geq \min\left\{ \frac{\e_{p}}{\rho_f + \Lambda\rho_g},\sqrt{\frac{2\e_{c}}{\Lambda\rho_g}}\right\}.$
    The above two inequalities imply the following bound on the number of steps:
    \begin{align*}
      K & \leq  \frac{2(\phi_0 - \phi^*)}{\rho_f+\mu_K +\max_k\lambda_{k+1}\rho_g}  \frac{1}{\min_k\|x_k - x_{k-1}\|^2}\\ 
      &\leq  
      \frac{2(\phi_0 - \phi^*)}{\rho_f+\mu_K +\Lambda\rho_g}
      \frac{1}{ \min\left\{ \frac{\e_{p}^2}{(\rho_f + \Lambda\rho_g)^2},
      \frac{2\e_{c}}{\Lambda\rho_g}\right\}} \\
      & = O\left(\max\left\{\frac{1}{\e_{p}^2}, \frac{1}{\e_{c}}\right\}\right) .
    \end{align*}

   
\textit{Bounding sample complexity} Recall that, the sample complexity of Algorithm \ref{alg:base} to solve the problem up to accuracy $(\e_{p},\e_{c})$ according to Theorem \ref{theorem:str_cvx_conv} is $ \mathcal N_k = O(\frac{\Lambda_k}{\beta_k} + \frac{R_k^2}{\beta_k^2}\log\frac{1}{\e_{c}})O\left(\frac{\sigma^2}{\e_{c}^2}+\frac{(L + \sigma)^2}{\mu\e_{p}^2}\right) = \tilde O(\frac{1}{\beta_k}\left(\frac{1}{\e_{c}^2} + \frac{1}{\e_{p}^2}\right)).$ 
Then, the total number of samples would be \begin{align*}
& \mathcal N = K\mathcal N_k  = O\left( \max\left\{\frac{1}{\e_{p}^2}, \frac{1}{\e_{c}}\right\} \right)
\tilde O\left(\frac{\check\lambda_{k}}{\beta_k}+ \frac{1}{\mu_d}\right)O\left(\frac{1}{\e_{c}^2} + \frac{1}{\mu\e_{p}^2}\right),
\end{align*}
 where 
 $$\beta_k := -\min_{x\in\R^d}\{g(x) + \frac{\rho_g}{2}\|x - x_{k-1}\|^2\}.$$
In the next who lemmas, we lower bound $\mu_d$ and $ \beta_k$.
\begin{lemma}\label{lemma:mud_lb}
We can lower bound $\mu_d$ of a subproblem $\mathcal P_{x_{k-1}}$ as follows:
    $\mu_d \geq  \frac{\beta_k \mu_g}{32 (M_f + \lambda M_g)}$.
\end{lemma} 
\begin{proof}
    First, let us upper bound the diameter of the feasibility set for (\ref{sub-problem:nonconvex}). Let us define $\check x_{k} := \arg\min_{x\in\R^d} g(x) + \frac{\rho_g}{2}\|x - x_{k-1}\|^2$. For all feasible points $y$ at the boundaries of (\ref{sub-problem:nonconvex}) we have $\frac{\mu_f}{2}\|y - \check x_{k}\|^2 \leq g(y) + \frac{\rho_g}{2}\|y - x_{k-1}\|^2 - \min_x (g(x) + \frac{\rho_g}{2}\|x - x_{k-1}\|^2) \leq g(y) + \frac{\rho_g}{2}\|y - x_{k-1}\|^2 + \beta_k \leq \beta_{k} $,
    that implies $\forall y\in\mathcal X_{k}$ we have $\|y - \check x_{k}\|^2 \leq  \frac{2\beta_k}{\mu_g},$ that is $\forall y,z\in\mathcal X_{k}$ we have $ \|y - z\|^2 \leq 2\|y - \check x_{k}\|^2 +2\|y - \check x_{k}\|^2  \leq \frac{8\beta_k}{\mu_g},$ that is:
    $R_k^2\leq \frac{8\beta_k}{\mu_g}.$ Then, by using Lemma (\ref{lemma:local_dual_strong_concavity}:
    $ \mu_d\geq \frac{\beta_k^2}{4R_k^2(M_f + \lambda M_g)} \geq \frac{\beta_k \mu_g}{32 (M_f + \lambda M_g)}.$
\end{proof}

The lemma below allows to lower bound $\beta_k,$ 
under the regularity assumption similar to LB-SGD, extended MFCQ that we made above.

\begin{lemma}\label{lemma:betak_lb}
   Let Assumption \ref{assumption:mfcq} hold with constant $\theta>0$, then at step $k$ we can lower bound $\beta_k$ 
   as follows $\beta_k \geq \min\{\theta, \frac{\mu_g l^2}{2 \rho_g^2}\}.$ 
\end{lemma}
\begin{proof}
 Let us define $\check x_{k} := \arg\min_{x\in\R^d} g(x) + \frac{\rho_g}{2}\|x - x_{k-1}\|^2$. If $\hat x_k$ is at most $\theta$
 -close to the boundary $-g(\hat x_k) \leq \theta$, then it satisfies $\nabla g(\check x_{k}) + \rho_g (\check x_{k} - x_{k-1})  = 0$, what implies $\|\check x_{k} - x_{k-1}\| = \frac{\|\nabla g(\check x_{k})\|}{\rho_g} \geq \frac{l}{\rho_g}$ (by Assumption \ref{assumption:mfcq}).
 Then we have that the maximum value of the constraint at problem (\ref{sub-problem:nonconvex}) $\beta_k := \max_{x\in\R^d}\{-g(x) - \frac{\rho_g}{2}\|x - x_{k-1}\|^2\}$ is lower bounded by
 \begin{align*}
 \beta_k  &= -g(\check x_{k}) - \frac{\rho_g}{2}\|\check x_{k} - x_{k-1}\|^2 \\
 & \geq -g(x_{k-1}) + \frac{\mu_g}{2}\|\hat x_k - x_{k-1}\|^2
 \geq \frac{\mu_g l^2}{2 \rho_g^2}.
\end{align*}
In the alternative case, $\beta_k \geq \theta.$ In any case, we have 
\begin{align}
    \beta_k \geq \min\left\{\theta, \frac{\mu_g l^2}{2 \rho_g^2}\right\}.
\end{align} 
 \end{proof}

Combining two lemmas above, and putting their results into the bound of SCSA (Alg. \ref{alg:base}), we get: 
\begin{align*}
 \mathcal N & = K\mathcal N_k = O\left( \max\left\{\frac{1}{\e_{p}^2}, \frac{1}{\e_{c}}\right\} \right)
\tilde O\left(\frac{\check\lambda_{k}}{\beta_k}+ \frac{1}{\mu_d}\right)O\left(\frac{1}{\e_{c}^2} + \frac{1}{\mu\e_{p}^2}\right)\\
& =O\left( \max\left\{\frac{1}{\e_{p}^2}, \frac{1}{\e_{c}}\right\} \right)
\tilde O\left(\frac{\check\lambda_{k}}{\beta_k}+ \frac{32 (M_f + \check \lambda_k M_g)}{\beta_k \mu_g}\right)O\left(\frac{1}{\e_{c}^2} + \frac{1}{\mu\e_{p}^2}\right)\\
&= O\left( \max\left\{\frac{1}{\e_{p}^2}, \frac{1}{\e_{c}}\right\} \max \left\{\frac{1}{\theta}, \frac{2\rho_g^2}{\mu_g l^2}\right\}\right)
\tilde O\left(\check\lambda_{k}+ \frac{32 (M_f + \check \lambda_k M_g)}{ \mu_g}\right)O\left(\frac{1}{\e_{c}^2} + \frac{1}{\mu\e_{p}^2}\right).
\end{align*}
If $\theta$ is constant, or if we can make the bound it dependent on $\theta$, and given $\e_{c} = O(\e_{p}^2)$ (from Eq. (\ref{eq:e_k})), then the bound is $\tilde O(\frac{1}{\theta\e_{p}^6})$.
\end{proof}

\subsection{Proof of Lemma \ref{lemma:init_sgd}}\label{proof:lemma:init_sgd}
that at every approximate gradient step $\tau>0$ uses a mini-batch of $n_0(\tau) = \frac{4\sigma^2}{\|\nabla \tilde{\mathcal  L}(x_{\tau-1}, \check\lambda)\|^2}$ samples
      and the step-size $\gamma_{\tau}\leq \frac{1}{M_f + \check\lambda M_g}.$
 First, let us prove the descent.
 \begin{lemma}\label{lemma:init_disturb}
        Let the Lagrangian gradient estimator $\tilde \nabla \mathcal L(x,\lambda) = \nabla L(x,\lambda) + \xi$ (with $\xi = (1+\lambda)\zeta$ )
        be such that $\|\xi\|\leq \frac{\|\tilde \nabla \mathcal L(x,\lambda)\|}{2}$, then the stochastic gradient step with step size $\gamma \leq \frac{1}{M_f + \lambda M_g}$ implies descent.
        \end{lemma}
        
        \begin{proof}
        We consider steps defined by $x_{t+1} = x_t - \gamma_t \tilde \nabla L(x,\lambda).$ From smoothness:
            \begin{align*}
               L(x_{t+1},\lambda)& \leq L(x_{t},\lambda) + \la -\gamma_t \tilde \nabla L(x,\lambda), \nabla L(x,\lambda)\ra + \frac{M}{2}\gamma_t^2\|\tilde \nabla L(x,\lambda)\|^2\\
               L(x_{t+1},\lambda) - L(x_{t},\lambda) & \leq -\gamma_t  \left( \la \tilde \nabla L(x,\lambda), \tilde\nabla L(x,\lambda) - \xi\ra - \frac{1}{2}\|\tilde \nabla L(x,\lambda)\|^2\right)\\
               & \leq -\gamma_t  \left( \frac{1}{2}\|\tilde \nabla L(x,\lambda)\|^2 - \|\tilde \nabla L(x,\lambda)\| \| \xi\|\right).
            \end{align*}
            Using $\|\xi\|\leq \frac{\|\tilde \nabla L(x,\lambda)\|}{2}$ we get
            $$L(x_{t+1},\lambda) - L(x_{t},\lambda)  \leq -\gamma_t  \left( \frac{1}{2}\|\tilde \nabla L(x,\lambda)\|^2 - \|\tilde \nabla L(x,\lambda)\| \| \xi\|\right) \leq 0,$$ what concludes the proof.
        \end{proof} 
         To get such an estimation accuracy $\|\xi\|\geq \|\tilde \nabla L(x,\lambda)\|/2 $ until we reach $\|\tilde \nabla L(x,\lambda)\|\leq \e_p^0$, one requires $n_0 = \frac{\hat\sigma^2}{(\e_p^0)^2}$. Then, for convergence of such a descent method to accuracy $ \e_p^0$  for strongly-convex problems $T = \tilde O(\log \frac{1}{\e_p^0})$ steps are required, 
    that implies
    $\mathcal N_0 = O(\frac{1}{(\e_p^0)^2}\log \frac{1}{\e_p^0}).$
\subsection{Safety}
Note that for safety, $\check \lambda_{k}$ must be chosen so that $-g(\check x_{k}) 
- \frac{\rho_g}{2}\|\check x_{k} - x_{k-1}\|^2 > 0,$ and that descent on $\mathcal K_{x_{k-1}}(\cdot,\check \lambda_{k})$ implies feasibility subject to $g(x).$
\begin{lemma}
    If $\check \lambda_{k}$ is chosen such that $\check \lambda_{k}\geq \frac{2\eta_{k-1}}{\rho_g \|x_{k-1} - x_{k-2}\|^2} + \frac{\rho_f}{\rho_g} + (2-\frac{\mu_g}{\rho_g})\lambda_{k-1} $, then any descent method on $\mathcal K_{x_{k-1}}(\cdot, \check\lambda_{k})$ starting from $x_{k-1}$ guarantees feasibility of all its iterates $\{ x_{t}\}$ subject to $g(x).$ 
\end{lemma} 
\begin{proof}
 Let $\{x_t\}$ be a descent sequence for minimizing $\mathcal \mathcal K_{x_{k-1}}(x,\check \lambda_k)$ such that $\mathcal K_{x_{k-1}}(x_t,\check \lambda_k) \leq\mathcal K_{x_{k-1}}( x_{k-1},\check \lambda_k).$ That means:
    \begin{align}\label{eq:K_descent}
       f(x_t) + \frac{\rho_f}{2}\|x_t - x_{k-1}\|^2 &+ \check \lambda_k (g(x_t) + \frac{\rho_g}{2}\|x_t - x_{k-1}\|^2 )\leq f(x_{k-1}) + \check \lambda_k g( x_{k-1}).
    \end{align}
%
    
    About $x_{k-1}$ we know:
    \begin{align}
    & \lambda_{k-1}( g(x_{k-1}) + \frac{\rho_g}{2}\|x_{k-1} - x_{k-2}\|^2) \leq \e_{c,k-1},\\
    &\|\nabla f(x_{k-1}) + \rho_f(x_{k-1} - x_{k-2}) + \lambda_{k-1}(\nabla g(x_{k-1}) + \rho_g(x_{k-1} - x_{k-2}))\|\leq\e_{p,k-1}
    \end{align}
    and by strong-convexity and  $\eta_k$-optimality we have
    \begin{align}
    &f(x_{k-1}) + \frac{\rho_f}{2}\|x_{k-1} - x_{k-2}\|^2 + \lambda_{k-1} ( g(x_{k-1}) + \frac{\rho_g}{2}\|x_{k-1} - x_{k-2}\|^2) \\
    & \leq \min_{x}(f(x) + \frac{\rho_f}{2}\|x - x_{k-2}\|^2 + \lambda_{k-1} ( g(x) + \frac{\rho_g}{2}\|x - x_{k-2}\|^2)) +\eta_{k-1}\\
    &
    \leq  f(x_t) + \frac{\rho_f}{2}\|x_t - x_{k-2}\|^2 + \lambda_{k-1} ( g(x_t) + \frac{\rho_g}{2}\|x_t - x_{k-2}\|^2) +\eta_{k-1} - \frac{\mu_{k-1}}{2}\|x_t - x_{k-1}\|^2,
    \end{align}
    where $ \mu_{k-1}:= \mu_f + \lambda_{k-1}\mu_g$. 
    Hence, 
       \begin{align*}
        & f(x_{k-1}) - f(x_t)- \frac{\rho_f}{2}\|x_t - x_{k-1}\|^2  \leq  \underbrace{\frac{\rho_f}{2}\|x_t - x_{k-2}\|^2- \frac{\rho_f}{2}\|x_t- x_{k-1}\|^2  - \frac{\rho_f}{2}\|x_{k-1} - x_{k-2}\|^2}_{\frac{\rho_f}{2}(1)} \\
        & + \lambda_{k-1} ( g(x_t) + \frac{\rho_g}{2}\|x_t - x_{k-2}\|^2 -  g(x_{k-1}) - \frac{\rho_g}{2}\|x_{k-1} - x_{k-2}\|^2) + \eta_{k-1} - \frac{\mu_{k-1}}{2}\|x_t - x_{k-1}\|^2\\
        & \leq \frac{\rho_f}{2}(1) + \lambda_{k-1} (-  g(x_{k-1}) + \frac{\rho_g}{2}\|x_t - x_{k-1}\|^2) - \frac{\mu_{k-1}}{2}\|x_t - x_{k-1}\|^2\\
        & + \lambda_{k-1}( g(x_t) + \underbrace{\frac{\rho_g}{2}\|x_t - x_{k-2}\|^2  -\frac{\rho_g}{2}\|x_t - x_{k-2}\|^2  - \frac{\rho_g}{2}\|x_{k-1} - x_{k-2}\|^2}_{\frac{\rho_g}{2}(1)})  + \eta_{k-1}
        \end{align*}
        \begin{align}\label{eq:fs}
        & f(x_{k-1}) - f(x_t)- \frac{\rho_f}{2}\|x_t - x_{k-1}\|^2 \\
        & \leq \lambda_{k-1} (-  g(x_{k-1}) + \frac{\rho_g}{2}\|x_t - x_{k-1}\|^2) + \left(\frac{\rho_f}{2} + \lambda_{k-1}\frac{\rho_g}{2}\right)(1) +  \lambda_{k-1} g(x_{t})+\eta_{k-1} - \frac{\mu_{k-1}}{2}\|x_t - x_{k-1}\|^2.\nonumber
        \end{align}
        Recall from \cref{eq:K_descent} that $x_t$ is a descent sequence, so:
    \begin{align*}
       f(x_t) &+ \frac{\rho_f}{2}\|x_t - x_{k-1}\|^2 
       + \check \lambda_k (g(x_t) + \frac{\rho_g}{2}\|x_t - x_{k-1}\|^2 )\leq f(x_{k-1}) + \check \lambda_k g( x_{k-1}).\\
         \check \lambda_k g(x_t)   & \leq f(x_{k-1})  - f(x_t) - \frac{\rho_f}{2}\|x_t - x_{k-1}\|^2 + \check \lambda_k(g(x_{k-1}) - \frac{\rho_g}{2}\|x_t - x_{k-1}\|^2)\\
         & \leq \lambda_{k-1} (-  g(x_{k-1}) + \frac{\rho_g}{2}\|x_t - x_{k-1}\|^2) + \left(\frac{\rho_f}{2} + \lambda_{k-1}\frac{\rho_g}{2}\right)(1) \\
         &+  \lambda_{k-1} g(x_{t})+\eta_{k-1} -  \check \lambda_k(-g(x_{k-1}) +\frac{\rho_g}{2}\|x_t - x_{k-1}\|^2) - \frac{\mu_{k-1}}{2}\|x_t - x_{k-1}\|^2.
         \end{align*}
         where the second inequality follows from Eq. (\ref{eq:fs}). 
         That implies 
         \begin{align*}
          (\check \lambda_k -  \lambda_{k-1})g(x_{t})& \leq \eta_{k-1} +\left(\frac{\rho_f}{2} + \lambda_{k-1}\frac{\rho_g}{2}\right)(1) + (\check \lambda_k -  \lambda_{k-1})(g(x_{k-1}) - \frac{\rho_g}{2}\|x_t - x_{k-1}\|^2) - \frac{\mu_{k-1}}{2}\|x_t - x_{k-1}\|^2.
    \end{align*}
    \begin{align}\label{eq:g_bound}
         g(x_{t}) & \leq \frac{\eta_{k-1} +\left(\frac{\rho_f}{2} + \lambda_{k-1}\frac{\rho_g}{2}\right)(1)}{(\check \lambda_k -  \lambda_{k-1})} + g(x_{k-1}) - \frac{\rho_g}{2}\|x_t - x_{k-1}\|^2 - \frac{\mu_{k-1}}{2}\|x_t - x_{k-1}\|^2.
    \end{align}
    We want to guarantee feasibility of all the iterates, that is, $g(x_t) \leq 0$, i.e.,:
    $$\frac{\eta_{k-1} +\left(\frac{\rho_f}{2} + \lambda_{k-1}\frac{\rho_g}{2}\right)(1) - \frac{\mu_{k-1}}{2}\|x_t - x_{k-1}\|^2}{(\check \lambda_k -  \lambda_{k-1})} + g(x_{k-1}) - \frac{\rho_g}{2}\|x_t - x_{k-1}\|^2 \leq 0$$
    which holds when
    $$\check \lambda_k  -  \lambda_{k-1}
    \geq 
    \frac{\eta_{k-1} +\left(\frac{\rho_f}{2} + \lambda_{k-1}\frac{\rho_g}{2}\right)(1)- \frac{\mu_{k-1}}{2}\|x_t - x_{k-1}\|^2
    }{
    - g(x_{k-1}) + \frac{\rho_g}{2}\|x_t - x_{k-1}\|^2
    }.$$
    Therefore, it is enough to set
    $$\check \lambda_k  \geq \lambda_{k-1} +  
    \frac{
    \eta_{k-1} +\left(\frac{\rho_f}{2} + \lambda_{k-1}\frac{\rho_g}{2}\right)(\|x_{k-1} - x_{k-2}\|^2+\|x_t - x_{k-1}\|^2) - \frac{\mu_{k-1}}{2}\|x_t - x_{k-1}\|^2
    }{
    - g(x_{k-1}) + \frac{\rho_g}{2}\|x_t - x_{k-1}\|^2 
    },$$
    since $(1)=\|x_t - x_{k-2}\|^2 - \|x_t - x_{k-1}\|^2  - \|x_{k-1} - x_{k-2}\|^2  \leq \|x_{k-1} - x_{k-2}\|^2 + \|x_{t} - x_{k-1}\|^2.$ 
    Moreover, we know that $$- g(x_{k-1}) \geq  \frac{\rho_g}{2}\|x_{k-1} - x_{k-2}\|^2.$$
    Therefore, we can upper bound the required increase by: 
    \begin{align}
    \check \lambda_k - \lambda_{k-1}&  \geq
    \frac{
    \eta_{k-1} +\left(\frac{\rho_f}{2} + \lambda_{k-1}\frac{\rho_g}{2}\right)\|x_{k-1} - x_{k-2}\|^2+\left(\frac{\rho_f-\mu_f}{2} + \lambda_{k-1}\frac{\rho_g-\mu_g}{2} \right) \|x_t - x_{k-1}\|^2
    }{
    \frac{\rho_g}{2}\|x_{k-1} - x_{k-2}\|^2 + \frac{\rho_g}{2}\|x_t - x_{k-1}\|^2 
    }.
    \end{align}
     The above holds when 
    \begin{align}
   \check \lambda_k - \lambda_{k-1} & \geq \frac{2\eta_{k-1}}{\rho_g \|x_{k-1} - x_{k-2}\|^2} + \frac{\rho_f}{\rho_g} + \left(1 - \frac{\mu_g}{\rho_g}\right)\lambda_{k-1}.
    \end{align}
    Then, choosing $\check \lambda_{k} = (2-\frac{\mu_g}{\rho_g})\lambda_{k-1} + \frac{\rho_f}{\rho_g} + \frac{\eta_{k-1}}{-g(x_{k-1})}$, 
    we imply $g(x_t)\leq 0$ for all $\{x_t\}$ such that $\mathcal K_{x_{k-1}}(x_t, \check\lambda_{k}) \leq \mathcal K_{x_{k-1}}( x_{k-1},\check\lambda_{k}).$ 
    \end{proof}
    
\begin{lemma}
    If $\check \lambda_{k}$ is chosen such that $\check \lambda_{k}\geq \frac{\eta_{k-1}}{-g(x_{k-1})} + \frac{\check \eta_{k}}{-g(x_{k-1})}+\frac{\rho_f}{\rho_g} + 2\lambda_{k-1} $, and $\mu_f\geq \rho_f/2$, and $\mu_g\geq \rho_g/2,$ then for the $k$-th iterate, 
    any $\eta_k$-approximate minimizer
    $\check x_{k} \approx \arg\min_{x\in\R^d} \mathcal K_{x_{k-1}}(x, \check\lambda_{k})$ is a strictly feasible point subject to $g(x) + \frac{\rho_g}{2}\|x - x_{k-1}\|^2.$ 
\end{lemma}  
\begin{proof}
Recall first that from strong concavity of $\mathcal K_{k-1}(x,\check \lambda_{k})$ we get 
$$\frac{\hat \mu_{k}}{2}\|\check x_{k} - x_{k-1}\|^2 \leq \mathcal K_{k-1}(x_{k-1},\check \lambda_{k}) - \mathcal K_{k-1}(\check x_{k},\check \lambda_{k}) + \eta_k $$
    with $\hat \mu_k  := \mu_f + \check \lambda_k \mu_g $
 i.e., similarly to \cref{eq:K_descent} we have:
    \begin{align*}
       f(\check x_{k}) &+ \frac{\rho_f}{2}\|\check x_{k} - x_{k-1}\|^2 
       + \check \lambda_k (g(\check x_{k}) + \frac{\rho_g}{2}\|\check x_{k} - x_{k-1}\|^2 ) + \frac{\hat \mu_{k}}{2}\|\check x_{k} - x_{k-1}\|^2 \\
       &\leq f(x_{k-1}) + \check \lambda_k g( x_{k-1}) + \eta_k .\end{align*}
       Then, 
       \begin{align}\label{eq:middle_g}
         & \check \lambda_k (g(\check x_{k}) + \frac{\rho_g}{2}\|\check x_{k} - x_{k-1}\|^2 )\nonumber\\
         & \leq f(x_{k-1})  - f(\check x_{k}) - \frac{\rho_f}{2}\|\check x_{k} - x_{k-1}\|^2 + \check \lambda_k(g(x_{k-1})) - \frac{\hat \mu_k}{2}\|\check x_{k} - x_{k-1}\|^2 + \eta_k\nonumber\\
         & \leq \lambda_{k-1} (-  g(x_{k-1}) + \frac{\rho_g}{2}\|\check x_{k} - x_{k-1}\|^2) + \left(\frac{\rho_f}{2} + \lambda_{k-1}\frac{\rho_g}{2}\right)(1) - \frac{ \mu_{k-1}}{2}\|\check x_{k} - x_{k-1}\|^2\nonumber\\
         &+  \lambda_{k-1} g(\check x_{k})+\eta_{k-1}  -  \check \lambda_k(-g(x_{k-1}) ) + \eta_k - \frac{\hat \mu_k}{2}\|\check x_{k} - x_{k-1}\|^2.
         \end{align}
         where in the second inequality we used Eq. (\ref{eq:fs}),  and $(1) := \|x_t - x_{k-2}\|^2  -\|x_t - x_{k-2}\|^2  - \|x_{k-1} - x_{k-2}\|^2$.
Note that if $\mu_f\geq \rho_f/2$ and $\mu_g\geq \rho_g/2,$ then 
\begin{align}\label{eq:norms}
  & \left(\frac{\rho_f}{2} + \lambda_{k-1}\frac{\rho_g}{2}\right) ( \|x_t - x_{k-2}\|^2  -\|x_t - x_{k-2}\|^2  - \|x_{k-1} - x_{k-2}\|^2) - \frac{\mu_{k-1}}{2}\|x_t - x_{k-1}\|^2  - \frac{\hat \mu_{k}}{2}\|x_t - x_{k-1}\|^2 \nonumber \\
  & \leq   \left(\frac{\rho_f}{2} + \lambda_{k-1}\frac{\rho_g}{2}\right) ( \|x_t - x_{k-2}\|^2  -\|x_t - x_{k-2}\|^2  - \|x_{k-1} - x_{k-1}\|^2) - \mu_{k-1}\|x_t - x_{k-1}\|^2 \nonumber \\
  & \leq   \left(\frac{\rho_f}{2} + \lambda_{k-1}\frac{\rho_g}{2}\right) ( \|x_t - x_{k-2}\|^2  -2\|x_t - x_{k-1}\|^2  - \|x_{k-1} - x_{k-2}\|^2) \nonumber \\
  & \leq \left(\frac{\rho_f}{2} + \lambda_{k-1}\frac{\rho_g}{2}\right)\|x_{k-1} - x_{k-2}\|^2
\end{align}
         
Then combining \Cref{eq:norms} with \Cref{eq:middle_g} we have:
\begin{align}
    g(\check x_{k})+ \frac{\rho_g}{2}\|\hat x_k - x_{k-1}\|^2 & \leq \frac{\eta_{k-1} + \eta_k+\left(\frac{\rho_f}{2} + \lambda_{k-1}\frac{\rho_g}{2}\right)\|x_{k-1} - x_{k-2}\|^2}{(\check \lambda_k -  \lambda_{k-1})} + g(x_{k-1}) .
\end{align}
We require it to be strictly negative, that is:
\begin{align*}
    & \frac{\eta_{k-1}+ \eta_k +\left(\frac{\rho_f}{2} + \lambda_{k-1}\frac{\rho_g}{2}\right)\|x_{k-1} - x_{k-2}\|^2}{(\check \lambda_k -  \lambda_{k-1})} + g(x_{k-1}) \leq 0.\\
    & \check \lambda_k -  \lambda_{k-1} \geq \frac{\eta_{k-1} +\left(\frac{\rho_f}{2} + \lambda_{k-1}\frac{\rho_g}{2}\right)\|x_{k-1} - x_{k-2}\|^2}{-g(x_{k-1})}. 
\end{align*}
Similarly as before, note that $-g(x_{k-1})\geq \frac{\rho_g}{2}\|x_{k-1} - x_{k-2}\|^2$. 
Therefore, we require 
$$\check \lambda_{k} \geq 2\lambda_{k-1} + \frac{\eta_{k-1}+ \eta_k}{-g(x_{k-1})} + \frac{\rho_f}{\rho_g}.$$
\end{proof}


Then, we can conclude that all the iterates of Algorithm \ref{alg:non-convex} are feasible with high probability.
   Indeed,  safety of the internal updates follows from the safety of SCSA (Alg. \ref{alg:base}). The safety of the initialization runs follows from the above lemmas.
\end{document}